\documentclass[a4paper,10pt]{article}

\usepackage{amsthm}
\usepackage{thmtools}
\usepackage{hyperref}
\usepackage{mathrsfs}
\usepackage{comment}
\usepackage[margin=1.6in]{geometry}

\declaretheoremstyle[bodyfont=\sl]{slanted}

\swapnumbers
\declaretheorem[name=Definition,style=definition,qed=$\dashv$,
numberwithin=section]{dfn}
\declaretheorem[name=Definition,style=definition,numbered=no,qed=$\dashv$]{dfn*}
\declaretheorem[name=Definition,style=definition,numbered=no]{dfnnoqed*}

\declaretheorem[name=Theorem,style=slanted,sibling=dfn]{tm}
\declaretheorem[name=Theorem,style=slanted,numbered=no]{tm*}
\declaretheorem[name=Lemma,style=slanted,sibling=dfn]{lem}
\declaretheorem[name=Corollary,style=slanted,sibling=dfn]{cor}
\declaretheorem[name=Corollary,style=slanted,numbered=no]{cor*}
\declaretheorem[name=Remark,style=definition,sibling=dfn]{rem}
\declaretheorem[name=Question,style=definition,sibling=dfn]{ques}

\declaretheorem[name=Fact,style=definition,sibling=dfn]{fact}
\declaretheorem[name=Fact,style=definition,numbered=no]{fact*}

\swapnumbers
\declaretheoremstyle[headfont=\scshape]{claimstyle}
\declaretheorem[name=Claim,style=claimstyle]{clm}

\declaretheorem[name=Claim,style=claimstyle]{clmtwo}
\declaretheorem[name=Claim,style=claimstyle]{clmthree}

\declaretheorem[name=Claim,style=claimstyle,numbered=no]{clm*}

\declaretheorem[name=Subclaim,style=claimstyle,numbered=no]{sclm*}

\declaretheorem[name=Subsubclaim,style=claimstyle,numbered=no]{ssclm*}

\declaretheoremstyle[headfont=\scshape]{casestyle}

\newcommand{\grd}{\mathrm{grd}}

\newcommand{\Sigmatwostrong}{Suppose $\kappa$ is 
$\Sigma_2$-strong. Then 
$V_{\kappa+1}^{\Mmm_\kappa}=V_{\kappa+1}^\Mmm$.}
\newcommand{\textthmsharp}{Let $A$ be a set such that $A^\#$ exists. Let 
$\kappa$ be an $A$-indiscernible.
 Then $V_{\kappa+1}^{\Mmm^{L(A)}_\kappa}=V_{\kappa+1}^{\Mmm^{L(A)}}$
 and this set is wellordered in $\Mmm^{L(A)}_\kappa$.}
 \newcommand{\textthmmeas}{Let $\kappa$ be measurable and $\mu$ be a 
normal measure on $\kappa$.
 Then for $\mu$-measure one many $\gamma<\kappa$,
$\Mmm_\gamma\sats$``$V_{\gamma+1}$ is wellorderable''.}

\newcommand{\weakcompacttmwithout}{Let $\kappa$ be weakly compact. Then:
 \begin{enumerate}
  \item $\Mmm_\kappa\sats  \kappa$-$\DC$ +
``$\kappa$ is weakly compact''.\footnote{So also $\Mmm_\kappa\sats$``$\kappa^+$ 
is regular and $\her_{\kappa^+}\sats\ZFC^-$''.}
 \item for each 
$A\in\Mmm_\kappa\inter\her_{\kappa^+}$, 
$\Mmm_\kappa\sats$``$A$ is wellordered''. 
\footnote{Note that
 the ``$\kappa^+$'' and ``$\her_{\kappa^+}$'' here are computed in $V$, not 
$\Mmm_\kappa$.}
  \item if 
$\pow(\kappa)^{\Mmm_\kappa}$ has cardinality
  $\kappa$ then \tu{(}i\tu{)}
 $\kappa$ is measurable in $\Mmm_\kappa$,
 and \tu{(}ii\tu{)} $x^\#$ exists for every $x\in\pow(\kappa)^{\Mmm_\kappa}$,
 and $x^\#\in\Mmm_\kappa$.
\item
If $\Mmm_\kappa\sats$``$\mu$ is a countably complete ultrafilter over 
$\gamma\leq\kappa$'',
then the ultrapower $\Ult(\Mmm_\kappa,\mu)$ is wellfounded and the ultrapower 
embedding
\[ i^{\Mmm_\kappa}_\mu:\Mmm_\kappa\to\Ult(\Mmm_\kappa,\mu) \]
is fully elementary.
\end{enumerate}}

\newcommand{\weakcompacttm}{Let $\kappa$ be weakly compact. Then:
 \begin{enumerate}
  \item\label{item:kappa-choice} $\Mmm_\kappa\sats  \kappa$-$\DC$ +
``$\kappa$ is weakly compact''.\footnote{So also $\Mmm_\kappa\sats$``$\kappa^+$ 
is regular and $\her_{\kappa^+}\sats\ZFC^-$''.}
 \item\label{item:<kappa+-DC} for each 
$A\in\Mmm_\kappa\inter\her_{\kappa^+}$, 
$\Mmm_\kappa\sats$``$A$ is wellordered''. 
\footnote{Note that
 the ``$\kappa^+$'' and ``$\her_{\kappa^+}$'' here are computed in $V$, not 
$\Mmm_\kappa$.}
  \item\label{item:small_power} if 
$\pow(\kappa)^{\Mmm_\kappa}$ has cardinality
  $\kappa$ then \tu{(}i\tu{)}
 $\kappa$ is measurable in $\Mmm_\kappa$,
 and \tu{(}ii\tu{)} $x^\#$ exists for every $x\in\pow(\kappa)^{\Mmm_\kappa}$.
\item\label{item:ult_emb_elem}
If $\Mmm_\kappa\sats$``$\mu$ is a countably complete ultrafilter over 
$\gamma\leq\kappa$'',
then the ultrapower $\Ult(\Mmm_\kappa,\mu)$ is wellfounded and the ultrapower 
embedding
\[ i^{\Mmm_\kappa}_\mu:\Mmm_\kappa\to\Ult(\Mmm_\kappa,\mu) \]
is fully elementary.
\end{enumerate}}

\newcommand{\inacctm}{
Let $\kappa$ be inaccessible (so $\Mmm_\kappa\sats$``$\kappa$ is 
inaccessible''). Then:
 \begin{enumerate}
  \item\label{item:Mmm_kappa_kappa-amenc} $\Mmm_\kappa$ is 
$\kappa$-amenably-closed.
  \item\label{item:V_kappa_wo} 
$\Mmm_\kappa\sats$``$(\kappa,\her_\kappa)$-Choice''
 iff $\Mmm_\kappa\sats$``$V_\kappa$ is 
wellordered''.
 \item \label{item:<kappa-Choice}
$\Mmm\sats$``$(<\kappa,\her_{\kappa^+}
)$-Choice holds,
 and hence, $(\her_{\kappa^+})^{<\kappa}\sub\her_{\kappa^+}$''.
 \end{enumerate}}

 \newcommand{\inacctmwithout}{
If $\kappa$ be inaccessible then:
 \begin{enumerate}
 \item $\Mmm_\kappa\sats$``$\kappa$ is 
inaccessible'' and $\her_\kappa^{\Mmm_\kappa}=V_\kappa^{\Mmm_\kappa}\sats\ZFC$.
  \item $\Mmm_\kappa$ is $\kappa$-amenably-closed.
  \item
$\Mmm_\kappa\sats$``$(\kappa,\her_\kappa)$-Choice''
$\iff \Mmm_\kappa\sats$``$\her_\kappa$ is 
wellordered''.
 \item\label{item:4}
$\Mmm\sats$``$(<\kappa,\her_{\kappa^+}
)$-Choice holds,
 so $(\her_{\kappa^+})^{<\kappa}\sub\her_{\kappa^+}$''.
 \end{enumerate}}

\newcommand{\GCH}{\mathrm{GCH}}

\newcommand{\swsw}{\mathrm{swsw}}

\newcommand{\Mmm}{\mathscr{M}}

\newcommand{\myurl}{https://sites.google.com/site/schlutzenberg/home-1}
\newcommand{\ralfurl}{\ralfurlone\ralfurltwo}
\newcommand{\ralfurlone}{https://ivv5hpp.uni-muenster.de}
\newcommand{\ralfurltwo}{/u/rds/}

\newcommand{\lu}{https://www.uni-muenster.de/IVV5WS/WebHop/user/alietz/}
\usepackage{amsmath}
\usepackage{amsxtra}
\usepackage{amssymb}
\usepackage{lineno}
\usepackage{turnstile}
\usepackage{enumitem}
\usepackage[retainorgcmds]{IEEEtrantools}

\newcommand{\QQ}{\mathbb Q}
\newcommand{\PP}{\mathbb P}
\newcommand{\BB}{\mathbb B}
\newcommand{\sub}{\subseteq}

\newcommand{\all}{\forall}

\newcommand{\inter}{\cap}
\renewcommand{\int}{\inter}

\newcommand{\om}{\omega}
\newcommand{\pow}{\mathcal{P}}
\newcommand{\OR}{\mathrm{OR}}

\newcommand{\Tt}{\mathcal{T}}

\newcommand{\rg}{\mathrm{rg}}
\newcommand{\dom}{\mathrm{dom}}
\newcommand{\crit}{\mathrm{cr}}
\newcommand{\rest}{\!\upharpoonright\!}

\newcommand{\Ult}{\mathrm{Ult}}
\newcommand{\sats}{\models}
\newcommand{\elem}{\preccurlyeq}

\newcommand{\AC}{\mathrm{AC}}
\newcommand{\DC}{\mathrm{DC}}
\newcommand{\ZFC}{\mathrm{ZFC}}
\newcommand{\ZF}{\mathrm{ZF}}

\newcommand{\kappabar}{{\bar{\kappa}}}

\newcommand{\her}{\mathcal{H}}

\newcommand{\forces}{\dststile{}{}}

\DeclareMathOperator{\Th}{Th}
\DeclareMathOperator{\card}{card}
\DeclareMathOperator{\cof}{cof}

\newcommand{\psub}{\subsetneq}

\newcommand{\tu}{\textup}

\newcommand{\trcl}{trcl}
\title{Choice principles in local mantles}

\author{Farmer Schlutzenberg\footnote{Funded by the  Deutsche 
Forschungsgemeinschaft (DFG, German Research
Foundation) under Germany's Excellence Strategy EXC 2044-390685587,
Mathematics M\"unster: Dynamics-Geometry-Structure.}
\footnote{farmer.schlutzenberg@gmail.com, \url{\myurl}}}

\begin{document}

\maketitle

\begin{abstract}Assume $\ZFC$. Let $\kappa$ be a cardinal. A
\emph{${<\kappa}$-ground}
is  a transitive proper class $W$ modelling $\ZFC$ such that $V$ is a generic
extension of $W$ via a forcing $\PP\in W$ of cardinality
${<\kappa}$.
The \emph{$\kappa$-mantle} $\Mmm_\kappa$ is the intersection
of all ${<\kappa}$-grounds.

We prove that  certain partial choice principles in $\Mmm_\kappa$
are the consequence of $\kappa$ being inaccessible/weakly compact,
and some other related facts.
\end{abstract}

\section{Introduction}

Let us recall some standard notions from set-theoretic geology.
We generally assume $\ZFC$, though at times
(in particular in \S\ref{sec:grounds}) 
we will also consider a weaker theory $T_1$ (which includes $\AC$).

Given a transitive model $W$\footnote{Here we are not specific
about exactly what formalization of classes we use. We could
work in some  class set theory,
which allows quantification over such classes $W$, or we could
with more care restrict to classes definable from parameters;
for us a class must have the property
that the structure $(V,\in,W)$ satisfies $\ZFC$
in the language with symbols $\dot{\in},\dot{W}$ which interpret
$\in$ and $W$.} of $\ZFC$ and a 
forcing $\PP\in W$,
a \emph{$(W,\PP)$-generic} is a filter $G\sub\PP$
which is generic with respect to $W$.
For a cardinal $\kappa$, a \emph{${<\kappa}$-ground}
of $V$ is a transitive proper class $W\sats\ZFC$ such that
there is $\PP\in W$ with $\PP$ of cardinality ${<\kappa}$ (with cardinality 
as computed in $W$, or equivalently, in $V$) and a $(W,\PP)$-generic filter $G$
such that $V=W[G]$. A \emph{ground} is a $<\kappa$-ground
for some cardinal $\kappa$.\footnote{Throughout,
we consider only set-forcing, no class-forcing.} The \emph{mantle} $\Mmm$ is 
the intersection of all 
grounds.
The \emph{$\kappa$-mantle} $\Mmm_\kappa$ is the intersection of all 
${<\kappa}$-grounds.

By \cite{laver_vlc}, as refined in \cite{stgeol},
there is a formula $\varphi(x,y)$ in two free variables
such that (i) for all $r$, 
$W_r=\{x\bigm|\varphi(r,x)\}$
is a ground (possibly $W_r=V$), and (ii) for every ground
$W$ there is $r$ such that $W=W_r$.
Therefore we can discuss grounds uniformly,
and $\Mmm$ and $\Mmm_\kappa$ are transitive classes
which are definable ($\Mmm_\kappa$ from parameter $\kappa$).

In \S\ref{sec:grounds} we will give the proof of ground definability,
but from somewhat less than $\ZFC$: we show that it holds
under a certain theory $T_1$ (see \ref{dfn:T_1}),
which is true in $\her_\kappa$ whenever $\kappa$ is a strong
limit cardinal (assuming $\ZFC$). The proof is essentially
the usual $\ZFC$ proof, however.

From now on, we take $W_r$ to be defined as in 
\S\ref{sec:grounds}, by which $r=(\her_{\gamma^+})^W$
for some $\gamma\geq\om$ for which there is $\PP\in r$
and a $(W,\PP)$-generic $G$ with $W[G]=V$.

Let $\theta$ be a strong limit cardinal.
By Usuba \cite{usuba_ddg},
the grounds
are set-directed. By \cite{usuba_ddg} and \cite{schindler_fsttimtg},
this is moreover 
reasonably local, 
and in particular if $X\in\her_\theta$, then there is $s\in\her_\theta$
with
$W_s\sub\bigcap_{r\in X}W_r$. (For following Usuba's proof of \cite[Proposition 
5.1]{usuba_ddg},  note that we can take the regular
cardinal $\kappa$ of that proof with $\kappa<\theta$,
and then the model $W$ constructed there satisfies the $\kappa^{++}$-uniform 
covering property for $V$. Usuba then uses Bukovsky's theorem, \cite[Fact 
3.9]{usuba_ddg}, to deduce that $W$ is a ground of $V$. But
by \cite[Theorem 3.11]{schindler_fsttimtg}, the forcing for this can be taken 
of size
$2^{\kappa^{++}}$ in $W[g]=V$.)

Also by \cite{usuba_ddg}, $\Mmm\sats\ZFC$,
and by \cite[\S2]{usuba_extendible},
 $\Mmm_\theta\sats\ZF$ 
(so note $\Mmm_\theta\sats$``$\theta$ is a strong limit cardinal'',
in the $\ZF$ sense that $\Mmm_\theta$ has no surjection
$\pi:V_\alpha\to\theta$ with $\alpha<\theta$).
If
$V_\theta\elem_n V$ with $n$ large enough, then 
$V_\theta^{\Mmm_\theta}=V_\theta^{\Mmm}$, and hence
$V_\theta^{\Mmm_\theta}\sats\AC$.\footnote{\label{ftn:false_claim}An earlier
draft asserted here that 
(\cite{usuba_ddg}  shows) $\her_\theta^{\Mmm_\theta}=\her_\theta^{\Mmm}$
for every strong limit cardinal $\theta$. This
was however not used anywhere, and it is false, at least if $M_1$ 
exists; see \cite{local_mantles_in_Lx}.}
 Usuba  showed in \cite{usuba_extendible} that 
if $\kappa$ is an extendible 
cardinal then $\Mmm_\kappa=\Mmm$, so in this case, 
$\Mmm_\kappa\sats\ZFC$.
Hence Usuba asked in \cite{usuba_extendible} about whether
$\Mmm_\kappa\sats\ZFC$ in general. We consider
related questions in this paper. Let us first sketch some further history.

Suppose now $\kappa$ is inaccessible.
Then $V_\kappa^{\Mmm_\kappa}\sats\ZFC$.
For note that by inaccessibility and the remarks above,
for each $\alpha<\kappa$ there is some $r\in V_\kappa$
such that $V_\alpha^{W_r}=V_\alpha^{\Mmm_\kappa}$.
Since each $W_r\sats\ZFC$, it follows that 
$V_\kappa^{\Mmm_\kappa}\sats\ZFC$.
Clearly $\Mmm_\kappa\sats$``$\kappa$ is 
inaccessible'', and if $\kappa$ is Mahlo
then $\Mmm_\kappa\sats$``$\kappa$ is Mahlo''.

However, A. Lietz (\cite{lietz}) answered Usuba's question above negatively
(assuming large cardinals),
showing that in fact it is consistent
relative to a Mahlo cardinal that  $\kappa$ is Mahlo but 
$\Mmm_\kappa\sats$``$\kappa$-$\AC$ fails''. In fact, Lietz constructs
a forcing extension $L[G]$ of $L$ in which $\kappa$ is Mahlo and
$\Mmm_\kappa^{L[G]}$ satisfies ``there is a function 
$f:\kappa\to\her_{\kappa^+}$ for which there is no choice function''.
He also proved other related things.

In the last few years, the theory of Varsovian models
has also been developed by Fuchs, Schindler, Sargsyan and more recently the 
author.
Here, among other things, full mantles $\Mmm$ of certain fully
iterable mice  have 
been analyzed,
and shown to be strategy mice, hence satisfying $\ZFC$.
Analysis of natural $\kappa$-mantles of those mice was, however,  missing.
But using Varsovian model techniques,
the author then analyzed the $\kappa_0$-mantle of the mouse
$M_{\swsw}$ (Definition \ref{dfn:Mswsw}),
showing that it is a strategy mouse, modelling $\ZFC+\GCH$.
A very brief outline is given in \S\ref{sec:choice}
(but the other results in the note do not rely on this,
and no inner model theory 
appears elsewhere in the paper).
The argument
has elements in common with
 Usuba's extendibility
 proof.
 
 Schindler then 
showed that if $\kappa$ is measurable then $\Mmm_\kappa\sats\AC$,
hence $\ZFC$; see \cite{schindler_fsttimtg}.
In this note we adapt this
argument,
deducing that fragments of choice hold in $\Mmm_\kappa$
from the weak compactness and inaccessibility  of $\kappa$ respectively.

\begin{dfn}\label{dfn:(alpha,X)-Choice}
 Given an ordinal $\alpha$ and set $X$,
let \emph{$(\alpha,X)$-Choice} be the assertion
that for every function $f:\alpha\to X$,
there is a choice function for $f$.
And \emph{$({<\alpha},X)$-Choice}
is the assertion that $(\beta,X)$-Choice
holds for all $\beta<\alpha$.
\end{dfn}

Part \ref{item:4} of the following theorem applies to the kind
of functions involved in the failure of $\kappa$-$\AC$ in Lietz' example,
but now with domain ${<\kappa}$.
Note that we assume $\ZFC$ except where otherwise stated;
 \emph{$\kappa$-amenable-closure}
is defined in \ref{dfn:amenable-closure}.

\begin{tm*}[\ref{tm:inaccessible_tm}]
\inacctmwithout
\end{tm*}

\begin{rem}
In part \ref{item:<kappa-Choice},
the ``$\kappa^+$'' and ``$\her_{\kappa^+}$''
are both in the sense of $\Mmm_\kappa$.
However, it can be that $\kappa$ is Mahlo
 and $\Mmm_\kappa\sats$``$(\kappa,\her_{\kappa^+})$-Choice
 fails, and $(\her_{\kappa^+})^\kappa\not\sub\her_{\kappa^+}$'';
 indeed, note that this occurs in Lietz' example $L[G]$ mentioned above.
\end{rem}

In the following theorem, the initial observation
that $\Mmm_\kappa\sats$``$\her_\kappa$ is 
wellordered'' was due to Lietz:

\begin{tm*}[\ref{tm:weak_compact_tm}] \footnote{Regarding 
part \ref{item:<kappa+-DC},
the author initially
 observed that a variant
 of Schindler's argument gives that $\Mmm_\kappa\sats\kappa$-$\DC$,
 and then  Lietz and the author independently noticed
 that the argument can be adjusted to
 show that every
 set in $\her_{\kappa^+}\inter\Mmm_\kappa$ is wellordered
 in $\Mmm_\kappa$.}
\weakcompacttmwithout
\end{tm*}

As a corollary to Schindler's proof, one easily gets:

\begin{fact*}[\ref{cor:meas}]
\textthmmeas
\end{fact*}

As mentioned above, Usuba showed that $\Mmm=\Mmm_\kappa$ assuming $\kappa$ is 
extendible.
The next result indicates that there are signs of this
in the leadup to an extendible cardinal
(for the definition of a \emph{$\Sigma_2$-strong} cardinal,
see \ref{dfn:Sigma_2-strong}):

\begin{tm*}[\ref{tm:Sigma_2-strong}]
\Sigmatwostrong
\end{tm*}

Analogously, down lower:

\begin{tm*}[\ref{tm:sharp}]
\textthmsharp
\end{tm*}

For further related results, which involve
some inner model theory, see  \cite{local_mantles_in_Lx}.

Before beginning our discussion of these results,
we go through some background set-theoretic geology,
including a proof of the the definability of grounds
from a theory modelled by $\her_\kappa$ whenever
$\kappa$ is a strong limit cardinal.

\section{Grounds and mantles}\label{sec:grounds}
We discuss here some background, starting with the key
fact of the definability of set-forcing grounds under $\ZFC$,
proved by some combination of Laver, Woodin and Hamkins:

\begin{fact}
Let $M,N$ be proper class
transitive inner models of $\ZFC$
and $\gamma\in\OR$ with $\pow(\gamma)\inter M=\pow(\gamma)\inter 
N$.
Let $\PP\in M$ and $\QQ\in N$, with $\PP,\QQ\sub\gamma$,
and let $G$ be $(M,\PP)$-generic and $H$ be $(N,\QQ)$-generic
and suppose $M[G]=N[H]=V$. Then $M=N$.
\end{fact}

We will discuss the proof of the result above, for two purposes.
The fact is central to our concerns, and the proof
contains elements which will come up in various places later,
so it is natural to collect all these things together.
Second, we wish to prove a version which assumes less background theory 
(than $\ZFC$). The authors of \cite{bhtu} make use of an analysis
of the complexity of the definability of grounds.
As shown there, each ground $W$ is, in particular, $\Sigma_2$ in a parameter 
$r$. However, the $\Sigma_2$ definition given there is not particularly
local: to compute $V_\alpha^W$, they work in $V_\beta$,
for a significantly larger ordinal $\beta$. So for  \cite[Theorem 
4]{bhtu}, they adopt the background theory  $\ZFC_\delta$.
We show here that the ground definability can be done much more locally
(though still requiring $\Sigma_2$ complexity), hence requiring significantly 
less than $\ZFC_\delta$.

\begin{dfn}\label{dfn:T_1}
Let $T_1^-$ 
 be the following theory in the language of set theory.
The axioms are Extensionality, Foundation, Pairing, Union,
Infinity, ``Every set is bijectable with an ordinal'',
$\Sigma_1$-Separation
and $\Sigma_1$-Collection. Now let
\[ T_1=T_1^-+\text{ Powerset}.\qedhere\]
\end{dfn}

Note that $T_1^-\sats\AC$. We will show that models of $T_1$ can uniformly 
define 
their grounds from parameters. First we give some lemmas.

\begin{lem}
 Assume $\ZFC$. Then for every cardinal $\kappa\geq\om$, 
(i) $\her_\kappa\sats T_1^-$, and
(ii) $\her_\kappa\sats T_1$ iff $\kappa$ is a strong limit cardinal.
\end{lem}

The usual proofs from $\ZFC$ easily adapt to give:
\begin{lem}
Assume $T_1$. Then  (i) for each ordinal $\xi$,
 $\her_\xi$ exists, (ii) $V=\bigcup_{\xi\in\OR}\her_\xi$,
  (iii) $\her_\xi\elem_1 V$, (iv) $\her_\xi\sats T_1^-$,
  (iv) the Lowenheim-Skolem theorem holds.
\end{lem}

In the following lemma, the forcing relations $\forces_{\Sigma_i}$
for $i\in\{0,1\}$, and 
$\forces_{\Pi_1}$, are the relations defined in a first-order
manner over $M$ in the usual 
manner, and the strong-$\Sigma_{i+1}$-forcing relation
$\forces^*_{\Sigma_{i+1}}$ is the relation for which,
given a $\Pi_i$ formula $\psi(\vec{x},\vec{y})$
with free variables $\vec{x},\vec{y}$, and 
given $\vec{\tau}\in(M^\PP)^{<\om}$,
we say $p\forces^*_{\Sigma_{i+1}}\exists\vec{y}\ \psi(\vec{\tau},\vec{y})$
iff there is $\vec{\sigma}\in(M^\PP)^{<\om}$
such that $p\forces_{\Pi_i}\psi(\vec{\tau},\vec{\sigma})$.

\begin{lem}[Forcing over $T_1^-$ and $T_1$]\label{lem:T_1_forcing}
Let $M\sats T_1^-$. Let $\PP\in M$ be a poset with $\PP\sub\gamma\in\OR^M$
and $G$ be $(M,\PP)$-generic.
Then:
\begin{enumerate}
 \item\label{item:Sigma_0_fr} We have:
 \begin{enumerate}
 \item The $\Sigma_0$-forcing relation 
$\forces_{\Sigma_0}$ for $(M,\PP)$ is 
$\Delta_1^M(\{\PP\})$, uniformly.
\item The $\Sigma_1$-forcing relation $\forces_{\Sigma_1}$ for 
$(M,\PP)$ is 
$\Sigma_1^M(\{\PP\})$, uniformly.\footnote{In a previous
draft of this document,
it mistakenly said that the $\Sigma_1$-forcing relation
is $\Delta_1^M$-definable, which is clearly false,
since in the case of trivial forcing, it would imply
that $\Sigma_1^M=\Delta_1^M$.}
\item The $\Pi_1$-forcing relation $\forces_{\Pi_1}$ for $(M,\PP)$ is 
$\Pi_1^M(\{\PP\})$, uniformly.
\end{enumerate}
Hence, $\forces_{\Sigma_0}$, $\forces_{\Sigma_1}$ and $\forces_{\Pi_1}$ are
 absolute to $\her_{\kappa}^M$, for $M$-cardinals $\kappa>\gamma$.
\item The strong-$\Sigma_2$-forcing relation $\forces^*_{\Sigma_2}$
for $(M,\PP)$ is $\Sigma_2^M(\{\PP\})$, uniformly.
 \item\label{item:Sigma_1_ft} The forcing theorem for 
$\Sigma_0$, $\Sigma_1$, $\Pi_1$ formulas
 holds for  $M[G]$, with respect to $\forces_{\Sigma_0}$, $\forces_{\Sigma_1}$,
 $\forces_{\Pi_1}$; likewise for $\Sigma_2$ and $\forces^*_{\Sigma_2}$.
That is, if $\varphi$ is $\Sigma_i$, where $i\in\{0,1\}$, and 
$\vec{\tau}\in(M^\PP)^{<\om}$,
then
\[ M[G]\sats\varphi(\vec{\tau}_G)\iff \exists p\in G\ \Big[
M\sats\text{``}p\forces_{\Sigma_i}\varphi(\vec{\tau})\text{''}\Big].\]
Likewise for $\Pi_1$ with $\forces_{\Pi_1}$, and for $\Sigma_2$ with 
$\forces^*_{\Sigma_2}$.

 \item\label{item:M[G]_sats_T_1} $M[G]\sats T_1^-$,
 and if $M\sats T_1$ then $M[G]\sats T_1$.
 \item $M$ and $M[G]$ have the same
 cardinals $\kappa>\gamma$,
 \item for each $M$-cardinal $\kappa>\gamma$,
we have $\her_\kappa^{M[G]}=\her_\kappa^M[G]$.
\end{enumerate}
\end{lem}

Such
local forcing calculations are very common in the literature,
in particular in fine structure theory, where much more local calculations are 
often used. But we include a proof in case the reader has
not seen these before.
\begin{proof}
Parts \ref{item:Sigma_0_fr}, \ref{item:Sigma_1_ft} for $\Sigma_0$:
The usual internal 
definition of the $\Sigma_0$-forcing relation $\forces_0$
works locally; in fact,
for each $\xi\in\OR^M$ with $\xi\geq\gamma$,
the $\Sigma_0$-forcing
relation for names in $\her_\xi$, is $\Delta_1^{\her_\xi}(\{\PP\})$,
uniformly in $\xi$. This gives the Forcing Theorem
for $\Sigma_0$ formulas in the usual manner.

Parts  \ref{item:Sigma_0_fr}, \ref{item:Sigma_1_ft} for $\Sigma_1$:
We defined the strong-$\Sigma_1$-forcing relation $\forces^*_1$ over $M$ 
above. Using the $\Sigma_0$-Forcing Theorem, note that 
$M[G]\sats\exists y\ \varphi(y,\tau_G)$
iff there is $p\in G$ such that 
$M\sats$``$p\forces^*_1\exists y\ \varphi(y,\tau)$''.
Moreover, $\forces^*_1$ is uniformly $\Sigma_1^M(\{\PP\})$-definable.

Note that we take $\forces_1$ defined over $M$ as follows:
Working in 
$M$, for $\varphi$ being $\Sigma_1$
and $\tau\in M^\PP$, set
\[ p\forces_1\varphi(\tau) \iff \all q\leq p\ \exists r\leq q\ 
\Big[r\forces^*_1\varphi(\tau)\Big].\]

We claim that $p\forces_1\varphi(\tau)$
iff $p\forces_1^*\varphi(\tau)$.
For the non-trivial direction,
suppose $p\forces_1\varphi(\tau)$.
Then working in $M$, using
$\Sigma_1$-Collection and $\AC$,
we can put together a name $\sigma\in M^\PP$
showing that $p\forces^*_1\varphi(\tau)$.
This completes the calculation for $\Sigma_1$.

Parts  \ref{item:Sigma_0_fr}, \ref{item:Sigma_1_ft} for $\Pi_1$:
$\forces_{\Pi_1}$ is defined as usual: Working in $M$,
for $\varphi$ being $\Pi_1$ and $\tau\in M^\PP$,
say $p\forces_{\Pi_1}\varphi(\tau)$
iff there is no $q\leq p$ such that 
$q\forces_{\Sigma_1}\neg\varphi(\tau)$.
So $\forces_{\Pi_1}$ is $\Pi_1^M(\{\PP\})$.
If $p\in G$ and $p\forces_{\Pi_1}\varphi(\tau)$,
then clearly $M[G]\sats\varphi(\tau_G)$.
So suppose $M[G]\sats\varphi(\tau_G)$ where $\varphi$ is $\Pi_1$.
Let
\[ D=\{p\in\PP\bigm|p\forces_{\Sigma_1}\neg\varphi(\tau)\}.\]
By $\Sigma_1$-Separation, $D\in M$.
Let $D'=D\cup\{p\in\PP\bigm|\neg\exists q\in D\ [q\leq p]\}$,
then $D'\in M$, and since $D'$ is dense, this easily suffices.

Parts  \ref{item:Sigma_0_fr}, \ref{item:Sigma_1_ft} for $\Sigma_2$:
Here we only consider the strong-$\Sigma_2$ forcing relation
$\forces^*_{\Sigma_2}$, and the claims regarding this follow immediately
just like for $\forces^*_{\Sigma_1}$.

Part \ref{item:M[G]_sats_T_1}: Most of the axioms are routine consequences
of the previous parts.
Let us verify that $M[G]\sats\Sigma_1$-Collection.
Fix a $\Sigma_0$ formula $\varphi$ and $\sigma,\tau\in M^\PP$.
Let $t\in M$ be the transitive closure of $\{\sigma,\tau\}$.
Then there is $w\in M$ such that for all $p\in\PP$
and $\varrho\in t$,
if
\[ p\forces_{\Sigma_1}\text{``}\varrho\in\sigma\text{ and }\exists 
y\varphi(\varrho,\tau,y)\text{''},\]
then there is $y\in M^\PP\inter w$ such that
$p\forces_{\Sigma_0}\text{``}\varrho\in\sigma\text{ and 
}\varphi(\varrho,\tau,y)\text{''}$.
But then using $w$, we easily get a bound on witnesses in $M[G]$,
as desired. This and the $\Sigma_0$-Forcing Theorem
easily yields $\Sigma_1$-Separation.

The remaining parts follow from routine calculations with nice names.
\end{proof}

\begin{dfn}
 Let $(M,E)\sats T_1^-$. A \emph{ground} of $M$ is a
 $W\sub M$ such that:
 \begin{enumerate}
  \item $(W,E\rest W)$ is \emph{$M$-transitive};
 that is, for all $x\in W$ and all $y\in M$, if $yEx$ then $y\in W$,
  \item $W\sats T_1^-$,
  \item there is $\PP\in W$ and a $(W,\PP)$-generic
 $G\in M$ such that $M=W[G]$.
 \item If $(M,E)\sats T_1$ then $(W,E\rest W)\sats T_1$.\qedhere
 \end{enumerate}
\end{dfn}

We now prove that $T_1$ suffices for the definability of grounds
(in the sense of the definition above). 
The proof is essentially that due to some combination of Laver, Woodin and 
Hamkins.
In the proof we make implicit use of 
Lemma \ref{lem:T_1_forcing}, to allow the forcing calculations:

\begin{tm}[Ground definability under $T_1$]
Assume $T_1$.  Let 
$\gamma\in\OR$,  
$H\sub\her_{\gamma^+}$ and $\kappa\geq\gamma^+$ a cardinal.
Then there is at most
one transitive $M\sub\her_{\kappa}$ such that
$M\sats T_1^-$, $(\her_{\gamma^+})^M=H$, and
$M$ is a ground for $\her_{\kappa}$ via some  $\PP\in H$.
\end{tm}
\begin{proof}
We proceed by induction on $\kappa$.
For $\kappa=\gamma^+$ it is trivial.

Suppose $\kappa$ is a limit cardinal,
and that for each cardinal $\theta\in[\gamma^+,\kappa)$,
there is a (unique) model $M_\theta$ of ordinal height $\theta$
with the stated 
properties. Then clearly $M=\bigcup_{\theta<\kappa}M_\theta$
is the unique candidate at $\kappa$. To see that $M$ works,
we just need to verify that $M$ is indeed a set-ground
of $\her_\kappa$ via some $\PP\in H$;
i.e. there is $\PP\in H$ and an $(M,\PP)$-generic
$G\sub\PP$ such that $M[G]=\her_\kappa$. But we can use any $(\PP,G)$
which worked at some earlier $\theta$.
For let $\theta_0\leq\theta_1<\kappa$, and let 
$(\PP_0,G_0),(\PP_1,G_1)$ work for $M_0=M_{\theta_0}$ and $M_1=M_{\theta_1}$.
So
$G_0$ is also $(M_1,\PP_0)$-generic, and vice versa.
And since
$\her_{\gamma^+}^{M_0}=H=\her_{\gamma^+}^{M_1}$,
and $H[G_0]=\her_{\gamma^+}=H[G_1]$, it follows
that $\her_\kappa=M_0[G_0]=M_0[G_1]$ and $M_1[G_0]=M_1[G_1]=\her_\kappa$,
so the specific choice of $(\PP,G)$ is irrelevant.

So consider $\kappa=\theta^+>\gamma^+$. Let $M,N$
be grounds of $\her_\kappa$ with the stated properties. By induction,
$M\inter\her_\theta=N\inter\her_\theta$. It just
remains to verify that
$\pow(\theta)\inter M=\pow(\theta)\inter N$.
The proof is, however, not by contradiction;
we will not assume that $M\neq N$.
Fix $(\PP,G)$ such that $\PP\in H$ and $G$ is $(M,\PP)$-generic
and $M[G]=\her_\kappa$.

Suppose first that $\cof(\theta)>\gamma$, as this case is easier;
however, it is in the end subsumed into the general case.
Let $A\sub\theta$. Then:

\begin{clm} $A\in M$ iff $A\inter\alpha\in M$ for all $\alpha<\theta$.\end{clm}
\begin{proof}
For the non-trivial direction,
suppose $A\inter\alpha\in M$ for every $\alpha<\theta$.
Let $f:\theta\to M$ be 
$f(\alpha)=A\inter\alpha$.
Then $f\in\her_\kappa$. So
there is a $\PP$-name $\dot{f}\in M$ with $\dot{f}_G=f$.
Working in $M$, for $p\in\PP$, compute
\[ D_p=\{\alpha<\theta\bigm|\exists x\ 
[p\forces\dot{f}(\check{\alpha})=\check{x}]\},\]
and let
$f_p:D_p\to\theta$ be the function
\[ f_p(\alpha)=\text{ unique }x\text{ such that 
}p\forces\dot{f}(\check{\alpha})=\check{x}.\]
So $\left<D_p,f_p\right>_{p\in\PP}\in M$, and because $\cof(\theta)>\gamma$,
there is $p\in G$ such that $D_p$ is cofinal in $\theta$.
Then $f=\left(\bigcup_{\alpha\in D_p}f_p(\alpha)\right)\in M$.
\end{proof}

We now argue in general.

\begin{clm}
 Let $A\sub\theta$. Then $A\in M$ iff
 for every $X\in\pow(\theta)\inter M$
 such that $\card(X)<(\gamma^+)=(\gamma^+)^M$,
we have $A\inter X\in M$.
\end{clm}
\begin{proof}
The forward direction is trivial. So let $A\sub\theta$ with $A\notin M$.
Let $\dot{A}\in M$ be a $\PP$-name
and $p_0\in G$ such that 
$p_0\forces\dot{A}\sub\check{\theta}$.
For each $q\leq p_0$, \emph{if}
there is $\alpha<\theta$ such that
\[ q\not\forces\check{\alpha}\in\dot{A}\text{ and 
}q\not\forces\check{\alpha}\notin\dot{A},\]
then let $\alpha_q$ be the least such $\alpha$;
otherwise $\alpha_q$ is undefined.
Let $D$ be the set of all $q\leq p_0$
such that $\alpha_q$ exists.
Then $G\sub D$, because
otherwise $q$ decides all elements of $\dot{A}$,
so $A\in M$.

In $M$, let
 $X=\{\alpha_q\bigm|q\in D\}$.
Then  $X\in M$, $\card^M(X)\leq\gamma$ and
$X\inter A\notin M$, as desired. For given $Y\in\pow(X)\inter M$,
an easy density argument shows that $Y\neq X\inter A$.
\end{proof}

\begin{clm}Let $X\sub\theta$ with 
$\card(X)<\gamma^+$.
Then $X\in M$ iff $X\in N$.
\end{clm}
\begin{proof}
Suppose $X_0=X\in N$.
Let $\dot{X}\in M$ be a $\PP$-name for $X$.
Using the forcing relation and $\dot{X}$,
there is a set $X_1\in\pow(\theta)\inter M$
with $X_0\sub X_1$ and $\card(X_1)<(\gamma^+)^V$.
Proceeding back-and-forth, construct (in $V$)
a continuous sequence of sets $\left<X_\alpha\right>_{\alpha<\gamma^+}$ such 
that
(i) $X_0=X$,
(ii) $X_{\om\alpha+2n+1}\in M$ and $X_{\om\alpha+2n+2}\in N$, and
(iii) $\card(X_\alpha)<(\gamma^+)^V$.

Now $\gamma^+<\kappa$, so 
$\left<X_\alpha\right>_{\alpha<\gamma^+}\in\her_\kappa$,
so $M,N$ have names for this sequence.
So as in the $\cof(\theta)>\gamma$ case,
we get a cofinal set $D_M\sub\gamma^+$
such that $D_M\in M$ and $\left<X_\alpha\right>_{\alpha\in D_M}\in M$.
Likewise with a cofinal set $D_N\in N$.
Let $D'_M$ be the set of limit points of $D_M$,
and $D'_N$ likewise. So these are club in $\gamma^+$.
Let $\alpha\in D'_M\inter D'_N$. Then note
that
\[ X_\alpha=\left(\bigcup_{\beta\in 
D_M\inter\alpha}X_\beta\right)=
\left(\bigcup_{\beta\in D_N\inter\alpha}X_\beta\right)\in 
M\inter N.\]
Let $\pi:\xi\to X_\alpha$ be the increasing enumeration
of $X_\alpha$. Then $\xi<\gamma^+$ 
and  $\pi\in M\inter N$.
We have $X\sub\rg(\pi)$.
Let $\bar{X}=\pi^{-1}(X)$.
Then $\bar{X}\in N$. But $\her_{\gamma^+}^M=H=\her_{\gamma^+}^N$,
so $\bar{X}\in M$. So $\pi``\bar{X}=X\in M$, as desired.
\end{proof}

This completes the proof of  ground definability under $T_1$.
\end{proof}

\begin{rem}
If $M\sats T_1^-$+``there is a largest cardinal
 $\kappa$, and $\kappa$ is regular'', then grounds of
 $M$ via forcings $\PP$ of $M$-cardinality ${<\kappa}$
are also definable from parameters over $M$,
by arguing much as above.
\end{rem}

\begin{dfn}Assume $T_1$.
 Let $\varphi_{\mathrm{grd}}(r,x)$ be the formula ``
 there are $\gamma$, $\PP$, $G$, $M$, $\kappa$ such that
 $\gamma<\kappa$ are cardinals,
$M\sub\her_\kappa$ is transitive, $M\sats T_1^-$,
$\PP\in r=(\her_{\gamma^+})^M$,
$G$ is $(M,\PP)$-generic, $\her_\kappa=M[G]$ and $x\in M$''.
 
We write $W'_r=\{x\bigm|\varphi_{\mathrm{grd}}(r,x)\}$.
We say $r$ is a \emph{true index} iff $W'_r$ is proper class.
We write $W_r=W'_r$ for true indices $r$,
and $W_r=V$ otherwise.
\end{dfn}

\begin{cor}
 Assume $\ZFC+\GCH$ and let $\lambda$ be a limit cardinal.
 Then the grounds of $\her_\lambda$ are definable from parameters
 over $\her_\lambda$.
\end{cor}

\begin{rem}
 Assume $\ZFC+\GCH$. Then for each limit ordinal $\xi$,
$V_{\om+\xi}$ is equivalent in the codes to
 $\her_{\aleph_{\xi}}$.
 So one can correctly formulate ``grounds''
 of $V_{\om+\xi}$, and they are definable over $V_{\om+\xi}$ from parameters.
\end{rem}

So we have the standard uniform definability of grounds, just assuming $T_1$:
\begin{lem}
Let $M\sats T_1$. Then
$\{W_r^M\bigm|r\in M\}$
enumerates exactly the grounds of $M$ (with repetitions, including 
$M$ itself).\end{lem}

\begin{rem}Assume $T_1$. Note that $\varphi_{\grd}$ is $\Sigma_2$,
and the assertion ``$r$ is a true index'' is $\Pi_2$.
(In fact, there are fixed $\Sigma_2$ and $\Pi_2$ formulas,
such that $T_1$ proves that these fixed formulas always work.)
Moreover, letting $\xi=\card(\trcl(\{r,x\}))$, note that
$\varphi_{\mathrm{grd}}(r,x)$ is absolute between $V$ and
$\her_{(2^{\xi})^+}$. (It is witnessed
by some $(\her_{\xi^+},M)$,  a structure
of size $2^{\xi}$.) Therefore:\end{rem}

\begin{fact}[Local definability of grounds]\label{fact:local_ground_def}
Assume $T_1$+``There is a proper class of strong limit cardinals''.
Let $\lambda$ be a strong limit cardinal.
Let $r\in\her_\lambda$ be a true index.
Then $\her_\lambda\sats$``$r$ is a true index''
 and
 $W_r^{\her_\lambda}=W_r\inter\her_\lambda=\her_\lambda^{W_r}$.
\end{fact}

It seems it might be possible, however, that $\her_\lambda\sats$``$r$ is a true 
index'' while $r$ fails to be a true index in $V$.

The remaining facts in this section, and the rest of the paper,
have a background theory of $\ZFC$. We have not investigated
to what extent things go through under $T_1$.
By \cite[Proposition 5.1]{usuba_ddg} and an
examination of its proof, we have:

\begin{fact}[Local set-directedness of 
grounds (Usuba)]\label{fact:local_ground_directedness}
(Assume $\ZFC$.) Let $\theta$ be a strong limit cardinal and $R\in\her_\theta$. 
 Then there is
 $t\in\her_\theta$ with
 $t\in\bigcap_{r\in R}W_r$
 and $W_t\sub W_r$ and $W_t=W_t^{W_r}$ for each
 $r\in R$. In particular, $W_t\sub\bigcap_{r\in R}W_r$.
\end{fact}

\begin{proof}
We refer here to the \emph{$\lambda$-uniform covering property for $V$};
see  \cite[Definition 
4.2]{usuba_ddg} or \cite[Definition 2.1]{schindler_fsttimtg}.
Let us set up some of the notation from the proof of \cite[Proposition 
5.1]{usuba_ddg}. Let $X=R$ (following the notation
from \cite{usuba_ddg}).\footnote{We wrote $R$ in the 
statement of the fact for consistency with later notation.}We may assume 
that $X$ is a set of true indices $r$.
For $r\in X$ let $\PP_r\in W_r$ be a forcing witnessing that $r$ is a true 
index. 
Let $\kappa$ be a regular cardinal with $\kappa>\card(X)$
and $\kappa>\card(\PP_r)$ for each $r$ (so it suffices
if $\kappa>\card(\trcl(X))$). Then the proof of \cite[Proposition 
5.1]{usuba_ddg} constructs a ground $W\sub\bigcap_{r\in X}W_r$
with the $\lambda=\kappa^{++}$-uniform covering property for $V$.
Therefore by \cite[Theorem 3.3]{schindler_buk},
there is $\PP\in W$ such that $W\sats$``$\card(\PP)=2^{2^{<\lambda}}$''
and $W$ is a ground of $V$ via $\PP$. Let
$\gamma_0=\card^W(\PP)$ and $t_0=(\her_{\gamma_0^+})^W$.
So $\gamma_0<\theta$, $t_0$ is a true index and $W=W_{t_0}$.
Let $\BB\in W$ be such that $W\sats$``$\BB$ is the complete Boolean algebra 
determined by $\PP$''
(so $\PP$ is a dense sub-order of $\BB$).
So $\card^W(\BB)\leq(2^{\gamma_0})^W<\theta$.
Then by \cite[Lemma 15.43]{jech} (or \cite[Fact 
3.1]{usuba_ddg}) for each $r\in X$ there is some $\BB_r\in W$
with $\BB_r\sub\BB$ and there is a $(W,\BB_r)$-generic
$G_r$ such that $W[G_r]=W_r$.
So letting $\gamma=(2^{\gamma_0})^W$, then  $t=(\her_{\gamma^+})^W$
is as desired.
\end{proof}

An easy corollary of local set-directedness is:
\begin{fact}[Invariance of $\Mmm_\kappa$]\label{fact:mantle_invariance}
Let $\kappa$ be a strong limit cardinal
and $r\in\her_\kappa$.
 Then $\Mmm_\kappa^{W_r}=\Mmm_\kappa$. 
\end{fact}

\begin{lem}[Absoluteness of $\Mmm_\kappa$]\label{lem:ground_Sigma_2_elem}
Let $\kappa<\lambda$ be strong limit cardinals and suppose
$\her_\lambda=V_\lambda\elem_{2}V$. Then for each $r\in\her_\kappa$, we 
have:
\begin{enumerate}[label=\tu{(}\roman*\tu{)}]
\item\label{item:grounds_mantle_absolute} ${<\kappa}$-grounds and $\Mmm_\kappa$ 
are
absolute to $V_\lambda$: 
\[ W_r^{V_\lambda}=W_r\inter V_\lambda=V_\lambda^{W_r}\text{
and }\Mmm_\kappa^{V_\lambda}=\Mmm_\kappa\inter 
V_\lambda=V_\lambda^{\Mmm_\kappa},\]
\item\label{item:elem_ground} $V_\lambda^{W_r}\elem_2 W_r$,
\item\label{item:local_mantle_invariance} 
$\Mmm_\kappa^{V_\lambda^{W_r}}=\Mmm_\kappa^{W_r}\inter 
V_\lambda^{W_r}=\Mmm_\kappa\inter V_\lambda=\Mmm_\kappa^{V_\lambda}$.
\end{enumerate}
\end{lem}
\begin{proof}
 Part \ref{item:grounds_mantle_absolute}:
 The absoluteness of $W_r$ is because the class true indices
 $r$  is $\Pi_2$, and each $W_r$ is
  $\Sigma_2(\{r\})$. But then clearly
 \[ \Mmm_\kappa^{V_\lambda}=\bigcap_{r\in 
V_\kappa}W_r^{V_\lambda}=\bigcap_{r\in 
V_\kappa}V_\lambda^{W_r}=V_\lambda^{\Mmm_\kappa}.\]
 
 Part \ref{item:elem_ground}:
 If $W_r=V$ then this is trivial.
 Suppose $W_r\psub V$ and let $\varphi$ be $\Sigma_2$
 and $x\in W_r\inter V_\lambda$ and suppose that
 $W_r\sats\varphi(x)$. Then by Fact \ref{fact:local_ground_def},
 $V\sats\psi(x)$ where $\psi$
 asserts ``There is a strong limit cardinal $\xi$
 such that $W_r^{\her_\xi}\sats\varphi(x)$'',
 but this is also $\Sigma_2$, so $V_\lambda\sats\psi(x)$,
 so letting $\xi<\lambda$ witness this,
 again by Fact \ref{fact:local_ground_def}, we get $W_r\inter 
\her_\xi\sats\varphi(x)$,
 so $W_r\inter V_\lambda\sats\varphi(x)$.

 Part \ref{item:local_mantle_invariance}: This follows from the previous parts
 and Fact \ref{fact:mantle_invariance}.
\end{proof}

\begin{dfn}\label{dfn:amenable-closure}
Let $N$ be an inner model. Let $f:\kappa\to N$.
Say that $f$
is \emph{amenable to $N$} iff $f\rest\alpha\in N$ for every 
$\alpha<\kappa$. Say that $N$ is \emph{$\kappa$-amenably-closed}
 iff for every $f:\kappa\to N$, if $f$ is amenable to $N$ then $f\in N$.
Say that $N$ is \emph{$\kappa$-stationarily-computing}
(\emph{$\kappa$-unboundedly-computing})
iff for every $f:\kappa\to N$,
there is a stationary (unbounded) $A\sub\kappa$ such that $f\rest A\in N$.
\end{dfn}

\begin{lem}
Let $N$ be an inner model of $\ZF$ and $\kappa>\om$ be regular.
If $N$ is $\kappa$-stationarily-computing then $N$ is 
$\kappa$-unboundedly-computing.
If $N$ is $\kappa$-unboundedly-computing then $N$ is $\kappa$-amenably-closed.
\end{lem}
\begin{lem}\label{lem:ground_amenably-closed}
 Let $W$ be a $<\kappa$-ground of $V$,
 where $\kappa>\om$ is regular.
 Then $W$ is $\kappa$-stationarily-computing.
\end{lem}
\begin{lem}The intersection of any family of $\kappa$-amenably-closed
structures is  $\kappa$-amenably-closed.\end{lem}

\begin{lem}\label{lem:kappa-amenably-closed}
 If $\kappa$ is inaccessible
 then  $\Mmm_\kappa$ is $\kappa$-amenably-closed.
\end{lem}

\section{Choice principles in the $\kappa$-mantle}\label{sec:choice}

As mentioned above, from now on we have $\ZFC$ as background theory.

The first positive results along the lines of what we
will prove here (regarding about $\kappa$-mantles
when $\kappa<\infty$), consists in Usuba's work, 
including his extendibility 
result. This
was followed by Lietz' negative results \cite{lietz}.
Some time after
this, using the general theory
of \cite{vm2},
the author showed that the $\kappa_0$-mantle
$\Mmm^M_{\kappa_0}$ of $M=M_{\swsw}$ (see below) is a strategy mouse.
We 
give an outline of this argument, but it is primarily intended
for the reader familiar with inner model theory,
and can be safely skipped over, as the remainder
of the paper does not depend on it.
We omit all
specifics to do with Varsovian models, just mentioning
enough to indicate what is relevant here.
The full proof will appear in \cite{vm2}.

\begin{dfn}\label{dfn:Mswsw}
$M_{\swsw}$ denotes the least iterable proper class
mouse with ordinals 
$\delta_0<\kappa_0<\delta_1<\kappa_1$
satisfying ``each $\delta_i$ is Woodin and each $\kappa_i$ is strong''.
\end{dfn}

The Varsovian model analysis produces a mouse $M_\infty$,
which is the direct limit of (pseudo-)iterates $P$ of $M$
 via correct iteration trees $\Tt$ on $M$, with  
$\Tt\in M|\kappa_0$, and which are based on $M|\delta_0$.
It also defines a certain fragment $\Sigma$ of the iteration strategy
 for $M_\infty$, yielding a strategy mouse 
$M_\infty[\Sigma]$. It turns out that
$M_\infty[\Sigma]$ has universe $\Mmm^M_{\kappa_0}$.

What is relevant here is the proof that $\Mmm^M_{\kappa_0}\sub 
M_\infty[\Sigma]$.
 Let $X\in\Mmm^M_{\kappa_0}$ be a set of ordinals.
 We must see that $X\in M_\infty[\Sigma]$.
\footnote{What blocks the more obvious attempt to prove this
is that it is not clear
that the iteration maps $i_{PQ}$ between the iterates $P,Q$
of the direct limit system eventually fix $X$.}
Now $\kappa_0$ is measurable in $M$.
Let $E$ be a normal measure on $\kappa_0$,
in the extender sequence of $M$,
and let
\[ j:M\to U=\Ult(M,E) \]
be the ultrapower map. By elementarity,
$j(X)\in\Mmm_{j(\kappa_0)}^U$.
With methods from the Varsovian model analysis,
one can then construct a specific ${<j(\kappa_0)}$-ground $W$ of $U$,
with $W\sub M_\infty[\Sigma]$. So
\[ j(X)\in\Mmm_{j(\kappa_0)}^U\sub W\sub M_\infty[\Sigma].\]
Other facts from Varsovian model analysis give
$j\rest\alpha\in M_\infty[\Sigma]$
for each $\alpha\in\OR$.
But then  $X\in M_\infty[\Sigma]$, as desired, since
\[ \beta\in X\iff j(\beta)\in j(X).\]

The preceding argument has structural similarities to
Usuba's extendibility proof (see 
\cite{usuba_extendible}). Schindler then 
found the following result (see \cite{schindler_fsttimtg}). We will use an 
adaptation of the proof for Theorem 
\ref{tm:weak_compact_tm} later, so 
we present this one first as a warmup, and in order to note
a simple corollary. We give
essentially Schindler's proof, although the
precise implementation might differ 
slightly.

\begin{fact}[Schindler]\label{fact:meas} Let $\kappa$ be measurable.
 Then $\Mmm_\kappa\sats\AC$, so $\Mmm_\kappa\sats\ZFC$.
\end{fact}

\begin{proof}
 Let $A\in\Mmm_\kappa$. We will find a wellorder
${<_A}$ of $A$ with ${<_A}\in\Mmm_\kappa$.

Let $\mu$ be a normal measure on $\kappa$, $M=\Ult(V,\mu)$ and
$j=i^V_\mu:V\to M$
the ultrapower map.
So $\kappa=\crit(j)$
and $j(A)\in\Mmm_{j(\kappa)}^M$.

\begin{clmtwo}\label{clm:first_meas}
We have:
\begin{enumerate}
 \item\label{item:Mmm_sub} $\Mmm^M_{j(\kappa)}
\sub\Mmm_\kappa^M\sub\Mmm_\kappa$, and
\item\label{item:j_amenable} $j\rest\Mmm_\kappa$  is amenable to 
$\Mmm_\kappa$.
\end{enumerate}
\end{clmtwo}
\begin{proof}
Part \ref{item:Mmm_sub}: The first $\sub$ is immediate.
For the second, 
we 
have
\[ \Mmm_\kappa=\bigcap_{r\in V_\kappa}W_r \text{\ \ and\ \ } 
\Mmm_\kappa^M=\bigcap_{r\in V_\kappa}W^M_r. \]
Let $\mu_r=\mu\inter W_r$. Then by standard
forcing calculations and elementarity,
we get $\mu_r\in W_r$ and
\[ W^M_r=j(W_r)=\Ult(W_r,\mu)^V=\Ult(W_r,\mu_r)^{W_r}, \]
so $W^M_r\sub W_r$, so
$\Mmm_\kappa^M\sub\Mmm_\kappa$
as desired.

Part \ref{item:j_amenable}: 
Let $r\in V_\kappa$. Then calculations as above
give  $i^{W_r}_{\mu_r}\rest W_r\sub j$.
But $\Mmm_\kappa\sub W_r$, and so $j\rest\Mmm_\kappa$
is amenable to $W_r$. Therefore $j\rest\Mmm_\kappa$
is amenable to $\Mmm_\kappa$, as desired.
\end{proof}

Since $\kappa$ is a strong limit, Fact 
\ref{fact:local_ground_directedness} gives $s\in V_{j(\kappa)}^M$ such that
\[ \Mmm_{j(\kappa)}^M\sub W=W_s^M\sub\Mmm_\kappa^M.\]
So $j(A)\in W\sats\ZFC$, so there is a wellorder ${<^*}$ of $j(A)$
with ${<^*}\in W$.
But $W\sub\Mmm_\kappa^M$, so ${<^*}\in\Mmm_\kappa^M\sub\Mmm_\kappa$.

Now working in $\Mmm_\kappa$, where we have $k=j\rest A$
and $j(A)$ 
and ${<^*}$, we can define a wellorder $<_A$ of $A$ by setting,
for $x,y\in A$:
\[ x<_Ay\iff k(x)<^*k(y).\]
This completes the proof.
\end{proof}

As a corollary to the proof above, we observe:

\begin{cor}\label{cor:meas}
 \textthmmeas
\end{cor}
\begin{proof}
Continue with the notation from the proof of Fact \ref{fact:meas}.
We show $\Mmm_\kappa^M\sats$``$V_{\kappa+1}$ is wellorderable''.

\begin{clm*}
$V_{\kappa+1}\inter\Mmm^M_{j(\kappa)}=V_{\kappa+1}\inter
\Mmm_{\kappa}^M=V_{\kappa+1}\inter\Mmm_\kappa$.
\end{clm*}
\begin{proof} We have
$V_{\kappa+1}\inter\Mmm_\kappa\sub V_{\kappa+1}\inter\Mmm_{j(\kappa)}^M$
since
$j\rest\Mmm_\kappa:\Mmm_\kappa\to\Mmm_{j(\kappa)}^M$
is elementary and $\kappa=\crit(j)$. By Claim \ref{clm:first_meas}
of the proof of Fact \ref{fact:meas}, this 
suffices.
\end{proof}

By Fact \ref{fact:meas},
$\Mmm_\kappa\sats\AC$, so $\Mmm_{j(\kappa)}^M\sats\AC$ also.
Let ${<^*}\in\Mmm_{j(\kappa)}^M$ be a wellorder
of $V_{\kappa+1}\inter\Mmm_{j(\kappa)}^M$. Then 
${<^*}\in\Mmm_\kappa^M$ and ${<^*}$ is a wellorder of 
$V_{\kappa+1}\inter\Mmm_\kappa^M$.
\end{proof}

We next use the simple idea above to prove
that certain cardinals are ``stable'' with respect to the mantle.
The first observation is:

\begin{tm}\label{tm:sharp}
\textthmsharp
\end{tm}
\begin{proof}
Let $j:L(A)\to L(A)$ be elementary with $\crit(j)=\kappa$.
We write $\Mmm_\kappa$ for $\Mmm_\kappa^{L(A)}$;
likewise $\Mmm_{j(\kappa)}$. Now
$j\rest\Mmm_\kappa:\Mmm_\kappa\to\Mmm_{j(\kappa)}$
is elementary.
Clearly $\Mmm_{j(\kappa)}\sub\Mmm_\kappa$.
But also, $B=V_{\kappa+1}^{\Mmm_\kappa}\sub 
V_{\kappa+1}^{\Mmm_{j(\kappa)}}$
as in the previous proof. So $V_{\kappa+1}^{\Mmm_{j(\kappa)}}=B$.
But 
$V_{j(\kappa)}^{\Mmm_{j(\kappa)}}\sats\ZFC$,
so there is a wellorder of $B$ in 
$\Mmm_{j(\kappa)}\sub\Mmm_\kappa$.

It now follows that $V_{\kappa+1}^{\Mmm_\kappa}=V_{\kappa+1}^{\Mmm}$,
because we can take $j(\kappa)$ as large as we like,
hence past any true index.
\end{proof}

\begin{dfn}\label{dfn:Sigma_2-strong}
A cardinal $\kappa$ is \emph{$\Sigma_2$-strong} iff
 for every $\alpha\in\OR$ there is an elementary
 embedding $j:V\to M$ with $\alpha<j(\kappa)$ and $V_\alpha\sub M$ 
 and $\Th_{\Sigma_2}^M(V_\alpha)=\Th_{\Sigma_2}^V(V_\alpha)$.\footnote{That is, 
for each $\Sigma_2$ formula $\varphi$ and all $\vec{x}\in(V_\alpha)^{<\om}$,
we have $M\sats\varphi(\vec{x})$ iff $V\sats\varphi(\vec{x})$.}

An embedding $j:V\to M$ is \emph{superstrong}
iff $V_{j(\kappa)}\sub M$. A cardinal $\kappa$ is \emph{$\infty$-superstrong}
iff for every $\alpha\in\OR$ there is a superstrong embedding
$j$ with $\crit(j)=\kappa$ and $j(\kappa)>\alpha$.

A \emph{superstrong extender} is the $V_\beta$-extender
derived from a superstrong embedding $j:V\to M$
where $\beta=j(\kappa)$ and $\kappa=\crit(j)$.
\end{dfn}
Note that:
\begin{lem}
 If $E$ is a superstrong extender and $W\sats\ZFC$ is a transitive proper class 
with $E\in W$, then $W\sats$``$E$ is a superstrong extender''.
\end{lem}

\begin{rem}
 Say that a cardinal $\kappa$ is \emph{$\infty$-$1$-extendible}
 iff for every $\alpha\in\OR$ there is $\beta\in\OR$ with $\beta\geq\alpha$ and
 and an elementary
 $j:V_{\kappa+1}\to V_{\beta+1}$
 (hence $j(\kappa)=\beta$) with $\crit(j)=\kappa$.
\end{rem}

\begin{tm}
We have:
\begin{enumerate}
 \item\label{item:extendible_is_limit} Every extendible cardinal
 is $\infty$-$1$-extendible and carries a normal measure
 concentrating on $\infty$-$1$-extendible cardinals.
 \item\label{item:1-extendible_is_limit} Every $\infty$-$1$-extendible cardinal 
is 
 $\infty$-superstrong and carries a normal 
measure concentrating on $\infty$-superstrong cardinals.

\item\label{item:infty-ss_is_limit} Every $\infty$-superstrong cardinal is
$\Sigma_2$-strong and carries a normal measure concentrating on 
$\Sigma_2$-strong cardinals.
\end{enumerate}
\end{tm}
\begin{proof}
Part \ref{item:extendible_is_limit}:
This is routine and left to the reader.

Part \ref{item:1-extendible_is_limit}:
Let $\kappa$ be $\infty$-$1$-extendible.
Let $j:V_{\kappa+1}\to V_{\beta+1}$ be 
elementary
with $\crit(j)=\kappa$. Let $E$ be the extender
derived from $j$ with support $V_\beta$.
Let $M=\Ult(V,E)$ and $k:V\to M$ be the ultrapower map.
Then one can show that $k$ is a superstrong embedding
with $k(\kappa)=\beta$ and that 
$M\sats$``$\kappa$ is $\infty$-superstrong''.
(For the last clause, consider ultrapowers $\Ult(M,E\rest\alpha)$
where $\alpha<\beta$, and show that unboundedly many of these
produce superstrong embeddings in $M$,
and also use that $\beta$ is $1$-$\infty$-extendible in $M$.)

Part \ref{item:infty-ss_is_limit}: Let $\kappa$ be $\infty$-superstrong.
We show first that $\kappa$ is $\Sigma_2$-strong.
So let $\alpha\in\OR$. We may assume that $V_\alpha\elem_2 V$.
Let $j:V\to M$ be any superstrong embedding with $\crit(j)=\kappa$ and 
$\alpha<j(\kappa)$. It suffices to verify:
\begin{clmthree}
$\Th_{\Sigma_2}^M(V_\alpha)=\Th_{\Sigma_2}^V(V_\alpha)$.
\end{clmthree}
\begin{proof}
Let $\varphi$ be $\Sigma_2$ and $\vec{x}\in(V_\alpha)^{<\om}$.
If $V\sats\varphi(\vec{x})$ then $V_\alpha\sats\varphi(\vec{x})$,
which implies $M\sats\varphi(\vec{x})$. Conversely,
suppose $M\sats\varphi(\vec{x})$. Because $\kappa$
is $\infty$-superstrong, it is clearly strong,
which implies that $V_\kappa\elem_2 V$. Therefore $V_{j(\kappa)}^M\elem_2 M$.
Therefore $V_{j(\kappa)}^M\sats\varphi(\vec{x})$.
But $V_{j(\kappa)}^M=V_{j(\kappa)}$, so $V_{j(\kappa)}\sats\varphi(\vec{x})$,
so $V\sats\varphi(\vec{x})$, as desired.\end{proof}

Now let $j:V\to M$ be a superstrong embedding with $\crit(j)=\kappa$. We will 
show that
$M\sats$``$\kappa$ is $\Sigma_2$-strong'', which completes the proof.

\begin{clmthree} $M\sats$``$\kappa$ is ${<\beta}$-$\Sigma_2$-strong'',
where $\beta=j(\kappa)$. That is, for each $\alpha<\beta$,
$M$ has an elementary $k:M\to N$ with $\crit(k)=\kappa$
and $V_\alpha\sub N$ and 
$\Th_{\Sigma_2}^N(V_\alpha)=\Th_{\Sigma_2}^M(V_\alpha)$.\end{clmthree}
\begin{proof}Since $M\sats$``$\beta$ is strong'',
$V_\beta^M\elem_2 M$ and there are club many $\alpha<\beta$
such that $V_\alpha^M=V_\alpha\elem_2M$. Fix some such $\alpha$.
Let $E_\alpha$ be the extender derived from $j$
with support $V_\alpha$.
Then $E_\alpha\in V_\beta\sub M$, and $M\sats$``$E_\alpha$ is an extender''.
Moreover, letting $N_\alpha=\Ult(M,E_\alpha)$, we have $V_\alpha\sub 
N_\alpha$
and
\[ 
\Th_{\Sigma_2}^{N_\alpha}(V_\alpha)=\Th_{\Sigma_2}^M(V_\alpha).\]
For let
$t=\Th_{\Sigma_2}^V(V_\kappa)=\Th_{\Sigma_2}^M(V_\kappa)$.
Then letting $k_\alpha:M\to N_\alpha$ be the ultrapower map,
\[ j(t)=\Th_{\Sigma_2}^M(V_\beta)\text{ and 
}k_\alpha(t)=\Th_{\Sigma_2}^N(V_{k_\alpha(\kappa)}^N). \]
So
$\Th_{\Sigma_2}^M(V_\alpha)=j(t)\inter V_\alpha=k_\alpha(t)\inter 
V_\alpha=\Th_{\Sigma_2}^N(V_\alpha)$.
\end{proof}

Now since $\kappa$ is $\Sigma_2$-strong, $M\sats$``$\beta=j(\kappa)$ is 
$\Sigma_2$-strong''. So let $\alpha\in\OR$ be a strong limit cardinal. Then $M$
has an embedding $\ell:M\to N$ with $\crit(\ell)=\beta$
and $V_\alpha^M=V_\alpha^N$ and 
$\Th_{\Sigma_2}^M(V_\alpha^M)=\Th_{\Sigma_2}^N(V_\alpha^M)$.
By the claim and elementarity, $N\sats$``$\kappa$ is 
${<\ell(\beta)}$-$\Sigma_2$-strong''. But then extenders
in $N$ which witness ${<\alpha}$-$\Sigma_2$-strength
in $N$ also witness this in $M$. Since $\alpha$ was arbitrary, we are done.
\end{proof}

We now prove an analogue of Usuba's extendibility result down lower:

\begin{tm}\label{tm:Sigma_2-strong}
\Sigmatwostrong
\end{tm}
\begin{proof}
 Suppose not and let $r$ be such that $V_{\kappa+1}^{W_r}\psub 
V_{\kappa+1}^{\Mmm_\kappa}$. Let $\lambda\in\OR$ be such 
that $\beth_\lambda=\lambda$ 
and $r\in V_\lambda$. Let $j:V\to M$ 
witness $\Sigma_2$-strength
with respect to $\lambda$.

Since the class of true indices is $\Pi_2$,
$M\sats$``$r$ is a true index''.
Also, by the local definability of grounds,
\[W_r^M\inter V_\lambda=W_r^{V_\lambda^M}=W_r^{V_\lambda}=W_r\inter V_\lambda.\]
In particular, $V_{\kappa+1}^{W_r^M}=V_{\kappa+1}^{W_r}\psub 
V_{\kappa+1}^{\Mmm_\kappa}$.

Since $r\in V_\lambda\sub V_{j(\kappa)}^M$, therefore
$\Mmm_{j(\kappa)}^M\inter 
V_{\kappa+1}\psub\Mmm_\kappa\inter V_{\kappa+1}$.
But since $\crit(j)=\kappa$, 
as in the proof of Theorem \ref{fact:meas}, we have
\[ \Mmm_\kappa\inter 
V_{\kappa+1}\sub\Mmm_{j(\kappa)}^M\inter V_{\kappa+1},\]
a contradiction.
\end{proof}

\begin{ques} Suppose $\kappa$ is strong. Is 
$V_{\kappa+1}^{\Mmm}=V_{\kappa+1}^{\Mmm_\kappa}$?\end{ques}

We now move toward the positive results
in the cases that $\kappa$ is inaccessible and/or weakly compact.
Toward these we first prove a couple of lemmas.

\begin{lem}[$\kappa$-uniform hulls]\label{lem:kappa-uniform_hull}
Let $\kappa$ be inaccessible. For true indices
$r\in V_\kappa$, let $(\PP_r,G_r)$ witness this,
and otherwise let $\PP_r=G_r=\emptyset$.
Let $\lambda=\beth_\lambda$ with $\cof(\lambda)>\kappa$ and $V_\lambda\elem_2 
V$.
Let $S\in V_\lambda$. Then there is $\widetilde{X}$ such that,
letting $\widetilde{X}_r=\widetilde{X}\inter V_\lambda^{W_r}$ for $r\in 
V_\kappa$, we have:
\begin{enumerate}
\item $V_\kappa\cup\{S,\kappa\}\sub \widetilde{X}\elem V_\lambda$
and $\widetilde{X}^{<\kappa}\sub \widetilde{X}$
and $|\widetilde{X}|=\kappa$, 
\item $\widetilde{X}_r\in W_r$ and $\widetilde{X}_r\elem V_\lambda^{W_r}\elem_2 
W_r$,
\end{enumerate}
and letting $X$ be the transitive collapse
of $\widetilde{X}$ and $\sigma:X\to \widetilde{X}$ the 
uncollapse
and $X_r,\sigma_r$ likewise, then:
\begin{enumerate}[resume*]
 \item $X_r\sub X$ and in fact, 
$X_r=W_r^{X}$,
 \item $\sigma:X\to V_\lambda$ is fully elementary with
  $\crit(\sigma)>\kappa$,
   \item $\sigma_r:X_r\to V_\lambda^{W_r}$ is fully elementary
   and $\sigma_r\sub\sigma$,
 \item $G_r$ is $(X_r,\PP_r)$-generic and 
$X=X_r[G_r]$,
 \item\label{item:hull_mantle_invariance} 
$\Mmm_\kappa^{X}=\Mmm_\kappa^{X_r}=\bigcap_{
s\in 
V_\kappa}X_s$; hence 
$\Mmm_\kappa^{X}\in\Mmm_\kappa$,
\item 
\label{item:<kappa-closure}
$X^{<\kappa}\sub X$ and 
$X_r^{<\kappa}\inter W_r\sub X_r$
and $(\Mmm_\kappa^{X})^{<\kappa} 
\inter\Mmm_\kappa\sub
\Mmm_\kappa^{X}$,
\item 
$\sigma\rest\Mmm_\kappa^{X}=\sigma_r\rest\Mmm_\kappa^{\bar{
\widetilde{X}} _r}$;
hence $\sigma\rest\Mmm_\kappa^{X}\in\Mmm_\kappa$,
\item 
$\sigma\rest\Mmm_\kappa^{X}:\Mmm_\kappa^{X}
\to\Mmm_\kappa^{V_\lambda
}$ is fully elementary.
\item\label{item:statements_in_X} 
$V_\lambda,\widetilde{X},X,\widetilde{X}_r,X
_r$ 
each satisfy $T_1$ and the following 
statements:
\begin{enumerate}
\item ``There are 
unboundedly many $\eta$
such that $\eta=\beth_\eta$'',
\item ``Fact \ref{fact:local_ground_directedness}'',
\item 
``There is 
$\xi=\beth_\xi$ such that for each $r\in V_\kappa$
and $s\in V_\kappa^{W_r}$,
we have $W_r\sats$``$s$ is a  index'' iff $V_\xi^{W_r}\sats$``$s$ is a true 
index''.
\end{enumerate}
\end{enumerate}
\end{lem}

\begin{proof}The fact that $V_\lambda^{W_r}\elem_2 W_r$
is by Lemma \ref{lem:ground_Sigma_2_elem}.

Construct an increasing sequence
$\left<\widetilde{X}_\alpha\right>_{\alpha<\kappa}$ such that
 $\widetilde{X}_\alpha\elem V_\lambda$
and $V_\kappa\cup\{x\}\sub \widetilde{X}_\alpha$
and $\widetilde{X}_\alpha^{<\kappa}\sub \widetilde{X}_\alpha$
and $|\widetilde{X}_\alpha|=\kappa$,
and such that for each $r\in V_\kappa$
there are cofinally many $\alpha<\kappa$
such that $\widetilde{X}_\alpha\inter W_r\in W_r$.

To construct this sequence, suppose we have
constructed $\widetilde{X}_\alpha$,
and let $r\in 
V_\kappa$. 
Let $X=\widetilde{X}_\alpha\inter W_r$.
By elementarity, $X\elem V_\lambda^{W_r}$ and
\[ 
\widetilde{X}_\alpha=X[G_r]=\{\tau_{G_r}\bigm|\tau\in\bar{
\widetilde{X}}\} .\]
Since $|X|=\kappa$, there is some $\widetilde{X}'\in 
W_r$ with $|\widetilde{X}'|=\kappa$ (hence 
$W_r\sats$``$|\widetilde{X}'|=\kappa$''),
and $X\sub \widetilde{X}'$, so there is also 
$\widetilde{X}''\in W_r$
with $\widetilde{X}''\elem V_\lambda^{W_r}$
and $\widetilde{X}'\sub \widetilde{X}''$ and $|\widetilde{X}''|=\kappa$ (in $V$ 
and $W_r$)
and such that $W_r\sats$``$(\widetilde{X}'')^{<\kappa}\sub (\widetilde{X}'')$''.
It easily follows that
\[ \widetilde{X}_\alpha\sub \widetilde{X}''[G_r]=\{\tau_{G_r}\bigm|\tau\in 
\widetilde{X}''\}\elem 
V_\lambda \]
and 
$\widetilde{X}''[G_r]\inter W_r=\widetilde{X}''$.
We set $\widetilde{X}_{\alpha+1}=\widetilde{X}''[G_r]$.
Then everything is clear except for the requirement
that $\widetilde{X}_{\alpha+1}^{<\kappa}\sub \widetilde{X}_{\alpha+1}$.
So let $f:\gamma\to \widetilde{X}_{\alpha+1}$ where $\gamma<\kappa$ (with $f\in 
V$);
we claim that $f\in \widetilde{X}_{\alpha+1}$.
Let $g:\gamma\to \widetilde{X}''$ be such that $g(\alpha)_{G_r}=f(\alpha)$ for 
each 
$\alpha<\gamma$.
So $g\in V$, but we don't know that $g\in W_r$.
But there is a $\PP_r$-name $\dot{g}\in V_\lambda^{W_r}$
such that $\dot{g}_{G_r}=g$. And $\widetilde{X}''\in W_r$, so there is $p_0\in 
G_r$
forcing that $\rg(\dot{g})\sub \widetilde{X}''$. Working in $W_r$ then,
we may fix for each $\alpha<\gamma$ an antichain
$A_\alpha\sub\PP_r$ maximal below $p_0$ and for each $p\in A_\alpha$
some $\tau_{\alpha p}\in \widetilde{X}''$ such that $p$ forces
that $\dot{g}(\alpha)=\tau_{\alpha p}$. 
Then the sequence $\left<\tau_{\alpha p}\right>_{(\alpha,p)\in I}$,
where
\[ I=\{(\alpha,p)\bigm|\alpha<\gamma\text{ and }p\in A_\alpha\},\]
is $\sub \widetilde{X}''$, and hence in $\widetilde{X}''$. Since 
$W_r\sats$``$(\widetilde{X}'')^{<\kappa}\sub (\widetilde{X}'')$'',
this gives
a name $\dot{g}''\in \widetilde{X}''$ such that
$p_0$ forces $\dot{g}''=\dot{g}$, and therefore
\[ g=\dot{g}_{G_r}=\dot{g}''_{G_r}\in 
\widetilde{X}''[G_r]=\widetilde{X}_{\alpha+1}. \]
But since $G_r\in \widetilde{X}_{\alpha+1}$, therefore $f\in 
\widetilde{X}_{\alpha+1}$,
so $\widetilde{X}_{\alpha+1}^{<\kappa}\sub \widetilde{X}_{\alpha+1}$ as desired.
With some simple bookkeeping then, we get an appropriate sequence.

Let now 
$\widetilde{X}=\bigcup_{\alpha<\kappa}\widetilde{X}_\alpha$.
We claim that $\widetilde{X}$ is as desired. The only thing we need to verify 
is that for each $r\in V_\kappa$, we have
\[ \widetilde{X}_r=\widetilde{X}\inter W_r\in W_r. \]

Fix $r$.  There is a $\PP_r$-name $\tau\in W_r$
such that $\tau_{G_r}=\left<\widetilde{X}_\alpha\right>_{\alpha<\kappa}$, and 
for cofinally 
many $\alpha<\kappa$
there is $p_\alpha\in G_r$ and $\widetilde{X}^r_\alpha\in W_r$ such that
\[ p_\alpha\forces\tau_\alpha\inter W_r=\check{\widetilde{X}^r_\alpha} \]
(hence $\widetilde{X}^r_\alpha=\widetilde{X}_\alpha\inter W_r$).
Since $\PP_r\in V_\kappa$, there is therefore a fixed $p\in\PP_r$
such that $p_\alpha=p$ for cofinally many $\alpha$. So
$\widetilde{X}_r=\bigcup_{\alpha\in I}\widetilde{X}^r_\alpha $
where
\[ I=\{\alpha<\kappa\bigm|\exists x\ [p\forces\tau_\alpha=\check{x}]\},\]
so $\widetilde{X}_r\in W_r$.

This completes the construction. The verification of the remaining
properties is now straightforward.
We omit discussing them, other than two remarks.
In part \ref{item:<kappa-closure},
the third statement follows directly from the first two
together with part \ref{item:hull_mantle_invariance};
the first two follow readily from the construction.
And in part \ref{item:statements_in_X}, note that $\xi$ exists
because $\cof(\lambda)>\kappa=|V_\kappa|$.\end{proof}

\begin{fact}
 Let $\kappa$ be weakly compact.
 Then $X$ be transitive with $\kappa\in X$ and $X^{<\kappa}\sub X$
 and $|X|=\kappa$.\footnote{In an earlier draft,
 the hypothesis ``$|X|=\kappa$'' was accidentally omitted,
 which obviously makes the statement equivalent
 to measurability.}
Then there is a non-principal $X$-$\kappa$-complete $X$-normal\footnote{That 
is, $\kappa$-completeness and normality with respect to
sequences in $X$.}
ultrafilter
$\mu$ over $\kappa$ such that letting $Y=\Ult(X,\mu)$
and $i^X_\mu$ the ultrapower embedding, then $Y$ is wellfounded. Moreover, 
$i^X_\mu$ is $\Sigma_1$-elementary and cofinal and $\crit(i^X_\mu)=\kappa$.
\end{fact}
\begin{proof}

 Let $\pi:X\to Z$ be any elementary embedding with $Z$ transitive
 and $\crit(\pi)=\kappa$. Let $\mu$ be the normal measure derived from $\pi$.
Note that $\mu$ works.
\end{proof}

We now extend the situation above, adding the assumption that 
$\kappa$ is weakly compact.

\begin{lem}[$\kappa$-uniform weak compactness 
embedding]\label{lem:kappa-uniform_wc}
Adopt the assumptions and notation from the statement
and proof of Lemma \ref{lem:kappa-uniform_hull}.
Assume further that $\kappa$ is weakly compact.
Let $\pi:X\to Y$
witness the weak compactness of $\kappa$ in $V$,
with $Y=\Ult(X,\mu)$ for an $X$-$\kappa$-complete $X$-normal ultrafilter
$\mu$ over $\kappa$, and $\pi=i^X_\mu$.
For $r\in V_\kappa$, 
let $\mu_r=\mu\inter X_r$.
Then:
\begin{enumerate}
\item\label{item:mu_r_in_W_r} $\mu_r\in W_r$ and $\mu_r$ is an 
$X_r$-$\kappa$-complete ultrafilter over $\kappa$;
let
\[ Y_r=\Ult(X_r,\mu_r)\text{ and }\pi_r:X_r\to Y_r \]
the ultrapower map; so $Y_r,\pi_r\in W_r$,
\item\label{item:mu_is_upward_closure} $\mu$ is the $X$-ultrafilter
generated by $\mu_r$ ($\mu_r$ is dense in $\mu$).
\item\label{item:function_representation}
For each $f:\kappa\to X_r$ with $f\in X$, there
is
$f_r\in X_r$ with $f_r:\kappa\to X_r$
and $f_r(\alpha)=f(\alpha)$ for $\mu$-measure one many 
$\alpha<\kappa$.\footnote{A draft
assumed only $f:\kappa\to X$, not $f:\kappa\to X_r$,
which obviously makes the statement false when $X\neq X_r$.}
 \item\label{item:Los_theorem} The ultrapowers satisfy
 Los' theorem 
for $\Sigma_1$ formulas, and $\pi_r,\pi$ are $\Sigma_2$-elementary.
 \item\label{item:Y_r_is_ground} $Y,Y_r\sats T_1$ and
 $Y_r$ is transitive, 
$Y_r=W_r^{Y}$, and $Y=Y_r[G_r]$.
\item\label{item:pi_r_sub_pi}$\pi_r\sub\pi$.
\item\label{item:pi(kappa)_mantles_invariant} 
$\Mmm_{\pi(\kappa)}^Y=\Mmm_{\pi_r(\kappa)}^{Y_r}\in W_r$;
hence this belongs to $\Mmm_\kappa$.
\item\label{item:pi_rest_in_mantle} 
$\pi\rest\Mmm_\kappa^X:\Mmm_\kappa^X\to\Mmm_{\pi(\kappa)}^Y$ is cofinal
$\Sigma_1$-elementary; this map belongs to $\Mmm_\kappa$.
\item\label{item:kappa_mantles_invariant} $\Mmm_\kappa^Y=\bigcap_{s\in 
V_\kappa}W^Y_s=\Mmm_\kappa^{Y_r}\in W_r$; hence this belongs to $\Mmm_\kappa$.
\item\label{item:ground_between_kappa_and_pi(kappa)}
$Y,Y_r$ each satisfy $T_1$ and the following statements:
\begin{enumerate}
\item\label{item:ground_between_a} ``There are 
unboundedly many $\eta$
such that $\eta=\beth_\eta$'',
\item\label{item:ground_between_b} ``Fact \ref{fact:local_ground_directedness} 
holds at $\theta=\pi(\kappa)=\beth_{\pi(\kappa)}$'',
\item\label{item:ground_between_c}
``There is 
$\xi=\beth_\xi$ such that for each $r\in V_{\pi(\kappa)}$ and $s\in 
V_{\pi(\kappa)}^{W_r}$,
we have $W_r\sats$``$s$ is true'' iff $V_\xi^{W_r}\sats$``$s$ 
is true''.
\end{enumerate}
Therefore there is $t\in V_{\pi(\kappa)}^Y$ with $W_t^Y\sub\Mmm_\kappa^Y$.
\end{enumerate}
\end{lem}

\begin{proof}
Parts \ref{item:mu_r_in_W_r}--\ref{item:function_representation}:
These are simple variants of the version for measurable cardinals $\kappa$
of $V[G]$ via  small forcing $\PP\in V_\kappa$;
one uses especially, however, the fact that $X_r$ is ${<\kappa}$-closed in 
$W_r$.
We leave the details to the reader.

Part \ref{item:Los_theorem}: Note that $V_\lambda$
satisfies $\Sigma_1$-Collection and ``For all $\alpha\in\OR$,
$V_\alpha$ exists and $\beth_\alpha\in\OR$ exists, and $\OR=\beth_\OR$'',
so $X_r,X$ do also.
Therefore if $\varphi$ is $\Sigma_0$ and $x\in X$
and
\[ X\sats\all\alpha<\kappa\ \exists y\ \varphi(x,y,\alpha) \]
then some $V_\xi^X\in X$ satisfies the same statement,
and hence there is $f\in X$ picking witnesses $y$.
This gives Los' theorem for $\Sigma_1$ formulas.
The $\Sigma_2$-elementarity of $\pi:X\to Y$ follows.
Likewise for $X_r,\pi_r$.

Parts \ref{item:Y_r_is_ground}, \ref{item:pi_r_sub_pi}:
The fact that $Y,Y_r\sats T_1$ follows from $\Sigma_2$-elementarity
and cofinality of $\pi,\pi_r$, and (for $\Sigma_1$-Collection)
that for each $\xi\in\OR^X$,
we have $\her_\xi^X\elem_1 X$ and $\her_\xi^{X_r}\elem_1 X_r$.
The rest follows as usual from the fact that
functions in $X$ with codomain $X_r$ are 
represented in $X_r$ (part \ref{item:function_representation}),
and again the $\Sigma_2$-elementarity of $\pi,\pi_r$.

Parts \ref{item:pi(kappa)_mantles_invariant}, \ref{item:pi_rest_in_mantle}:
Basically by invariance of $\Mmm_\kappa$  (Fact \ref{fact:mantle_invariance}),
we have $\Mmm_\kappa^X=\Mmm_\kappa^{X_r}$, and by part 
\ref{item:statements_in_X} of \ref{lem:kappa-uniform_hull},
 there is $\xi<\OR^X$ such that for each $r\in V_\kappa$
 and $s\in V_\kappa^{W_r}$, we have
$X_r\sats$``$s$ is true'' iff $V_\xi^{X_r}\sats$``$s$ is true''. Let
\[ T_r=\{s\in V_\kappa^{W_r}\bigm|W_r\sats\text{``}s\text{ 
 is true''}\}.\]
So $T_r\in X_r$ and has the same definition there;
likewise for $T_r\in Y_r$, since $\pi_r$ is $\Sigma_2$-elementary.
And because of the existence of $\xi$,
\[ \pi(T_r)=\{s\in V_{\pi_r(\kappa)}^{Y_r}\bigm|Y_r\sats\text{``}s\text{ 
is true''}\},\]
and it follows (in the case of $r=\emptyset$, but similarly in general),
\[ \Mmm_{\pi(\kappa)}^Y=\left(\bigcap_{s\in V_{\pi(\kappa)}^Y}W^Y_s\right)=
 \left(\bigcup_{\zeta\in I}
 \pi(\Mmm_\kappa^{V_\zeta})\right)\]
where $I$ is the set of all $\zeta\in[\xi,\OR^X)$
such that $\beth_\zeta^X=\zeta$. 
But $\Mmm_\kappa^X=\Mmm_\kappa^{X_r}$
and $\pi_r\sub\pi$, so
$\Mmm_{\pi(\kappa)}^Y=\Mmm_{\pi_r(\kappa)}^{Y_r}$.
The calculations above also show that
\[ \pi\rest\Mmm_\kappa^X:\Mmm_\kappa^X\to\Mmm_{\pi(\kappa)}^Y \]
is cofinal $\Sigma_1$-elementary, and likewise for $\pi_r\sub\pi$.

Part \ref{item:kappa_mantles_invariant}: By part \ref{item:Y_r_is_ground}, 
$W^Y_s=Y_s$, so $\Mmm_\kappa^Y=\bigcap_{s\in V_\kappa}W^Y_s$.
And note that the density of the grounds of $X_r$
in the grounds of $X$ is lifted
to that for those of $Y_r$ in those of $Y$.
(That is, for example, if $r,s$ are such that $X_r\sub X_s$,
then $Y_r\sub Y_s$, as this is preserved by $\pi$.)
So $\Mmm_\kappa^{Y_r}=\Mmm_\kappa^Y$, as desired.

Part \ref{item:ground_between_a}: For each $\zeta\in X$ with 
$\zeta=\beth_\zeta^X$,
we have $\pi(\zeta)=\beth_{\pi(\zeta)}^Y$.

Part \ref{item:ground_between_c}: If $\xi$ witnesses the
corresponding statement in $X$, note that 
$\pi(\xi)$ works in $Y$.

Part \ref{item:ground_between_b}:
We consider literally $Y$, but the same proof works for $Y_r$.
Note that there is a function $f:V_\kappa\to V_\kappa$
with $f\in X$, such that for each $R\in V_\kappa$,
$X\sats$``$t=f(R)$ is a true index
and $t$ witnesses Fact 
\ref{fact:local_ground_directedness}
for $R$'' ($f$ exists by the elementarity of $\sigma$).
We claim that $\pi(f)$ has the same property for $Y$.
For by $\Pi_2$-elementarity, $Y\sats$``Every $t\in\rg(\pi(f))$
is a true index''. Moreover, let $\xi$ be as before.
Then for each $\zeta$ such that $\xi<\zeta<\OR^X$ and $\zeta=\beth_\zeta^X$,
$V_\zeta^X$ satisfies ``$W_{f(R)}\sub W_r$ for each $R\in V_\kappa$
and $r\in R$''.
This lifts to $Y$ under $\pi$, and since $\pi$ is cofinal, this suffices.
\end{proof}

We are now ready to prove the main theorem for weakly compact $\kappa$.
The first proof that, under this assumption, $\Mmm_\kappa\sats$``$V_\kappa$ is 
wellordered''
is due to Lietz:

\begin{tm}\label{tm:weak_compact_tm}
\weakcompacttm
\end{tm}

\begin{proof}
 Part \ref{item:ult_emb_elem} follows directly from part 
\ref{item:kappa-choice}, as the wellfoundedness of $\Ult(\Mmm_\kappa,\mu)$
requires only $\om$-DC, and the 
proof of Los' theorem here
only uses $\kappa$-choice.
The conclusion that $x^\#$ exists in part \ref{item:small_power}
follows easily from the rest,
using the elementarity of $i_\mu$ and that $\Ult(\Mmm_\kappa,\mu)$ is 
wellfounded.
To see that $\Mmm_\kappa\sats$``$\kappa$ is weakly compact'',
let $T\sub{^{<\kappa}2}$ be a tree in $\Mmm_\kappa$.
Then $T$ has a cofinal branch $b$ in $V$, by weak compactness in $V$.
But $b\inter V_\alpha\in\Mmm_\kappa$ for each $\alpha<\kappa$.
Therefore by \ref{lem:ground_amenably-closed}, $b\in\Mmm_\kappa$.

Here is Lietz' argument that 
$\Mmm_\kappa\sats$``$V_\kappa$ is wellordered'':\footnote{The author first 
mistakenly
thought that a similar argument worked with $\kappa$ only inaccessible,
but Lietz noted that one seems to need weak compactness for this.}
Working 
in $\Mmm_\kappa$, let $T$ be the tree
of all attempts to build a wellorder of $V_\kappa$.
(For example, let $T\sub{^{<\kappa}V_\kappa}$ be the set of all functions
$f:\alpha\to V_\kappa$ where $\alpha<\kappa$,
such that for each $\beta<\alpha$,
$f(\beta)$ is a wellorder of $V_\beta$,
and for all $\beta_1<\beta_2<\alpha$, $f(\beta_2)$ is an end extension of 
$f(\beta_1)$.) Since $V_\kappa^{\Mmm_\kappa}\sats\ZFC$, $T$ is unbounded in 
$V_\kappa$, and clearly $T\rest\alpha\in V_\kappa$ for each $\alpha<\kappa$.
Therefore by weakly compactness in $\Mmm_\kappa$,  $\Mmm_\kappa$ has a 
$T$-cofinal 
branch, and clearly this gives a wellorder of $V_\kappa\inter\Mmm_\kappa$.

We proceed now to the proof that $\Mmm_\kappa\sats\kappa$-$\DC$,
and that every set $A\in\Mmm_\kappa\inter\her_{\kappa^+}$ is wellordered in 
$\Mmm_\kappa$.
Let $\mathscr{T}\in\Mmm$ be a $\kappa$-$\DC$-tree,\footnote{
That is, a set $\mathscr{F}$ of functions $f$
such that $\dom(f)<\kappa$, with $\mathscr{F}$ closed under initial segment,
and no maximal elements; that is, for every $f\in\mathscr{F}$
there is $g\in\mathscr{F}$ with $\dom(f)<\dom(g)$ and $f=g\rest\dom(f)$.
Note that $\kappa$-$\DC$ is just the assertion
that for every $\kappa$-$\DC$ tree $\mathscr{T}$,
there is a $\mathscr{T}$-maximal branch;
that is, a function $f\notin\mathscr{T}$ such that $f\rest\alpha\in\mathscr{T}$
for all $\alpha<\dom(f)$.} and let 
$A\in\Mmm_\kappa\inter\her_{\kappa^+}$.
Let $S=(\mathscr{T},A)\in V_\lambda$ and $X$ be a $\kappa$-uniform
hull, etc, with $S\in X$ and everything as in Lemma 
\ref{lem:kappa-uniform_hull},
and let $\pi:X\to Y$, etc, be as in Lemma \ref{lem:kappa-uniform_wc}.
So $\sigma:X\to V_\lambda$ is fully elementary with 
$\kappa<\crit(\sigma)$. Let $\sigma(\bar{\mathscr{T}})=\mathscr{T}$
and $\sigma(A)=A$.

By \ref{lem:kappa-uniform_wc}, 
$\pi'=\pi\rest\Mmm^X_\kappa:\Mmm^X_\kappa\to\Mmm^Y_{\pi(\kappa)}$ is cofinal 
$\Sigma_1$-elementary, and these models and map
belong to $\Mmm_\kappa$. We have $A,\bar{\mathscr{T}}\in\Mmm^X_\kappa$.

We first find a wellorder of $A$ in $\Mmm_\kappa$,
by arguing as in Schindler's proof of Fact \ref{fact:meas}, but using the weak 
compactness 
embedding.
We have $\pi'(A)\in\Mmm_{\pi(\kappa)}^Y$.
By \ref{lem:kappa-uniform_wc}, there is a ground $W$ of $\Mmm^Y_{\pi(\kappa)}$
such that
\[ \Mmm_{\pi(\kappa)}^Y\sub W\sub\Mmm_{\kappa}^Y\in\Mmm_\kappa.\]
So $W\sats\AC$ and $\pi'(A)\in W$.
Let ${<^*}\in W$ be a wellorder of $\pi'(A)$.
So ${<^*}\in\Mmm_\kappa$. Working in $\Mmm_\kappa$,
we can therefore wellorder $A$ by setting, for $x,y\in A$:
\[ x<_Ay\iff\pi'(x)<^*\pi'(y).\]

We now find a branch through $\bar{\mathscr{T}}$ in $\Mmm_\kappa$,
with length $\kappa$.
Let $B\in\Mmm_\kappa^X$ be the field of $\bar{\mathscr{T}}$.
As above, there is a wellorder
$<^*$ of $B$ in $\Mmm_\kappa$.
Working in $\Mmm_\kappa$,
we  recursively construct a sequence $\left<x_\alpha\right>_{\alpha<\kappa}$
constituting a branch through $\bar{\mathscr{T}}$,
using $<^*$ to pick next elements,
and noting that at limit stages $\eta<\kappa$,
we get $\left<x_\alpha\right>_{\alpha<\eta}\in\Mmm_\kappa^X$,
because
by \ref{lem:kappa-uniform_wc} part \ref{item:<kappa-closure}
we have $(\Mmm_\kappa^X)^{<\kappa}
\inter\Mmm_\kappa\sub\Mmm_\kappa^X$. By \ref{lem:kappa-uniform_hull},
$\sigma'=\sigma\rest\Mmm_\kappa^X\in\Mmm_\kappa$,
and note that
 $\left<\sigma'(x_\alpha)\right>_{\alpha<\kappa}$
is a cofinal branch through $\mathscr{T}$, as desired.

 Part \ref{item:small_power}: Now suppose
$\pow(\kappa)\inter\Mmm_\kappa\in\her_{\kappa^+}$.
Then we may assume that $A=\pow(\kappa)\inter\Mmm_\kappa$ above.
Therefore $\pi':\Mmm_\kappa^X\to\Mmm_{\pi(\kappa)}^Y$
is $\Mmm_\kappa$-total. Therefore $\kappa$ is measurable in $\Mmm_\kappa$.
Since  $\Mmm_\kappa\sats\kappa$-$\DC$,
the rest now follows, as discussed in the first paragraph of the proof. 
\end{proof}

Recall  $(\alpha,X)$-Choice
from Definition \ref{dfn:(alpha,X)-Choice}:

\begin{tm}\label{tm:inaccessible_tm}
\inacctm\end{tm}

\begin{rem}
Note that in part \ref{item:<kappa-Choice},
the ``$\kappa^+$'' and ``$\her_{\kappa^+}$''
are both in the sense of $\Mmm_\kappa$.
Note that also, as $\kappa$ is inaccessible,
$V_\kappa^{\Mmm_\kappa}\sats\ZFC$,
$\Mmm_\kappa\sats$``$\kappa$ is inaccessible'',
and $\Mmm_\kappa$ is $\kappa$-amenable closed,
by Lemma \ref{lem:kappa-amenably-closed}.
\end{rem}

\begin{proof}
Part \ref{item:Mmm_kappa_kappa-amenc} was Lemma \ref{lem:kappa-amenably-closed},
and since 
$V_\kappa^{\Mmm_\kappa}=\her_\kappa^{\Mmm_\kappa}\sats\ZFC$,
part  \ref{item:V_kappa_wo} is easy.

Part \ref{item:<kappa-Choice}: Let 
$\gamma<\kappa$ and $f\in\Mmm_\kappa$ with
$f:\gamma\to(\her_{\kappa^+})^{\Mmm_\kappa}$.
We find a choice function for $f$ in $\Mmm_\kappa$.
Write $f_\alpha=f(\alpha)$.
Fix a function $g:\gamma\to\Mmm_\kappa$ with
\[ g_\alpha=g(\alpha):\kappa\to 
\mathrm{trancl}(f_\alpha) \]
surjective
for each $\alpha<\gamma$.
Let $c:\gamma\to\Mmm_\kappa$ be $c_\alpha=c(\alpha)\sub\kappa$
the
induced code for $g_\alpha$
(so $c_\alpha,g_\alpha\in\Mmm_\kappa$,
but note we don't know that $c,g\in\Mmm_\kappa$).
Fix $\lambda$ and a $\kappa$-uniform hull
$\widetilde{X}\elem V_\lambda$ with
$f,c,g\in\widetilde{X}$
and everything else as
in \ref{lem:kappa-uniform_hull}.
So 
$\sigma(f,c,g)=(f,c,g)$.
Fix a club $C$ of $\bar{\kappa}<\kappa$
such that $\gamma<\bar{\kappa}$ and $V_{\bar{\kappa}}\elem V_\kappa$
and such that we get a corresponding system of 
structures $X^{\bar{\kappa}}_{r}$ and elementary embeddings
$\pi^{\bar{\kappa}}_r:X^{\bar{\kappa}}_r\to X_r$,
for $r\in V_{\bar{\kappa}}$,
with $X^{\bar{\kappa}}_r,\pi^{\bar{\kappa}}_r\in W_r$, $X^{\bar{\kappa}}_r$ of 
cardinality
$\bar{\kappa}$ in $W_r$,
 $\crit(\pi^{\bar{\kappa}}_r)=\bar{\kappa}$
and $\pi^{\bar{\kappa}}_r(\bar{\kappa})=\kappa$, and each 
$X^{\bar{\kappa}}_r[G_r]=X_{\emptyset\bar{\kappa}}$
and $\pi^{\bar{\kappa}}_r\sub\pi_{\emptyset\bar{\kappa}}$,
and with $f,c_\alpha,g_\alpha\in\rg(\pi^{\bar{\kappa}}_r)$
for each $\alpha<\gamma$.
Write 
$\pi^{\bar{\kappa}}_r(f^{\bar{\kappa}},
c^{\bar{\kappa}}_\alpha,g^{\bar{\kappa}}_\alpha)=
(f,c_\alpha,g_\alpha)$.
So $c^{\bar{\kappa}}_\alpha=c_\alpha\inter\bar{\kappa}$, 
so
$c^{\bar{\kappa}}_\alpha,g^{\bar{\kappa}}_\alpha\in
(\her_{\bar{\kappa}^+})^{\Mmm_\kappa}$,
and $f^{\bar{\kappa}}:\gamma\to(\her_{\bar{\kappa}^+})^{\Mmm_\kappa}$
with $f^{\bar{\kappa}}_\alpha\sub\rg(g^{\bar{\kappa}}_\alpha)$.
Let $c^{\bar{\kappa}}:\gamma\to\Mmm_\kappa$
be $c^{\bar{\kappa}}(\alpha)=c^{\bar{\kappa}}_\alpha$
and likewise for $g^{\bar{\kappa}}$.

In $V$,
pick a sequence $\left<{<_{\bar{\kappa}}}\right>_{\bar{\kappa}\in C}$
of wellorders $<_{\bar{\kappa}}$
of $(\her_{\bar{\kappa}^+})^{\Mmm_\kappa}$
with ${<_{\bar{\kappa}}}\in\Mmm_\kappa$.
Let  $z^{\bar{\kappa}}_\alpha$ be the $<_{\bar{\kappa}}$-least
element of $f^{\bar{\kappa}}_\alpha$,
and let $\xi^{\bar{\kappa}}_\alpha<\bar{\kappa}$ be the
least $\xi$ with $g^{\bar{\kappa}}_\alpha(\xi)=z_{\alpha}^{\bar{\kappa}}$.

Let $S$ be the stationary set of all strong limit cardinals
$\bar{\kappa}\in C$ of cofinality $\gamma^+$.
Enumerate ${^\gamma}\kappa$ as $\{s_\beta\}_{\beta<\kappa}$,
with ${^\gamma}\bar{\kappa}=\{s_\beta\}_{\beta<\bar{\kappa}}$
for each $\bar{\kappa}\in S$.
For $\bar{\kappa}\in S$, let $\beta_{\bar{\kappa}}$
be the $\beta<\bar{\kappa}$ such that 
$s_\beta=\left<\xi^{\bar{\kappa}}_\alpha\right>_{\alpha<\gamma
}$.
Let $S'\sub S$ be stationary and such that
$\beta_{\bar{\kappa}}$ is constant for $\bar{\kappa}\in S'$.

Let $d:\gamma\to\Mmm_\kappa$ be the choice function for $f$ given by
$d(\alpha)=\pi_{\emptyset\bar{\kappa}}(z^{\bar{\kappa}}_\alpha)$,
whenever $\bar{\kappa}\in S'$.
This is independent of $\bar{\kappa}\in S'$.
For if $\kappabar_0,\kappabar_1\in S'$ with $\kappabar_0<\kappabar_1$,
then for each $\alpha<\gamma$, we have
$\xi=\xi^{\bar{\kappa}_0}_\alpha=\xi^{\bar{\kappa}_1}_\alpha$,
so
\[ 
\pi_{\emptyset\bar{\kappa}_0}(z^{\bar{\kappa}_0}_\alpha)=
\pi_{\emptyset\bar{\kappa}_0}(g^{\bar{\kappa}_0}_\alpha(\xi))=
g_\alpha(\xi)=
\pi_{\emptyset\bar{\kappa}_1}(g^{\bar{\kappa}_1}_\alpha(\xi))=
\pi_{\emptyset\bar{\kappa}_1}(z^{\bar{\kappa}_1}_\alpha).\]

But $d\in\Mmm_\kappa$.
For given $r\in V_\kappa$, let $\bar{\kappa}\in S'$
with $r\in V_{\bar{\kappa}}$. 
Then $f^{\bar{\kappa}}\in X^{\bar{\kappa}}_r$ and 
$\pi^{\bar{\kappa}}_r(f^{\bar{\kappa}})=f$,
since $\pi^{\bar{\kappa}}_r\sub\pi_{\emptyset\bar{\kappa}}$.
And $X^{\bar{\kappa}}_r\in W_r$, so $f^{\bar{\kappa}}\in W_r$.
But 
$<_{\bar{\kappa}}$
is in $W_r$, and so 
$d^{\bar{\kappa}}=\left<z^{\bar{\kappa}}_\alpha\right>_{\alpha<\gamma}\in W_r$. 
And since
$\pi^{\bar{\kappa}}_r\sub\pi_{\emptyset\bar{\kappa}}$,
$\pi^{\bar{\kappa}}_r(d^{\bar{\kappa}})=d$. Since $\pi^{\bar{\kappa}}_r\in 
W_r$, 
therefore $d\in W_r$.
So $d\in\Mmm_\kappa\sats$``$d$ is a choice function for $f$'', so we are done.
\end{proof}

\bibliographystyle{plain}
\bibliography{../bibliography/bibliography}
\end{document}